\documentclass[11pt]{amsart}

\usepackage[utf8]{inputenc}
\usepackage[OT2,T1]{fontenc}
\usepackage[T1]{fontenc}
\usepackage[english]{babel}
\usepackage{lmodern}
\usepackage{textcomp}

\usepackage{amssymb, amsmath, amsthm}
\usepackage{fullpage}
\usepackage{graphicx}
\usepackage[all]{xy}

\usepackage{hyperref}
\usepackage[usenames,dvipsnames]{color}
\usepackage{enumerate}

\numberwithin{equation}{section}

\theoremstyle{plain}
\newtheorem{theorem}[equation]{Theorem}
\newtheorem{lemma}[equation]{Lemma}
\newtheorem{corollary}[equation]{Corollary}
\newtheorem{proposition}[equation]{Proposition}

\theoremstyle{definition}
\newtheorem{definition}[equation]{Definition}
\newtheorem{example}[equation]{Example}
\newtheorem{algorithm}[equation]{Algorithm}

\theoremstyle{remark}
\newtheorem{remark}[equation]{Remark}

\newcommand{\ZZ}{\mathbb{Z}}
\newcommand{\QQ}{\mathbb{Q}}

\newcommand{\PP}{\mathbb{P}}

\newcommand{\GL}{\mathop{\rm GL}\nolimits}
\newcommand{\PGL}{\mathop{\rm PGL}\nolimits}
\newcommand{\SL}{\mathop{\rm SL}\nolimits}
\newcommand{\Gal}{\mathop{\rm Gal}\nolimits}
\newcommand{\Br}{\mathop{\rm Br}\nolimits}

\newcommand{\p}{\mathfrak{p}}
\newcommand{\OO}{\mathcal{O}}
\newcommand{\Spec}{\mathop{\rm Spec}}

\newcommand{\Aut}{\mathop{\rm Aut}\nolimits}

\newcommand{\Char}{\mathop{\rm char}\nolimits}
\newcommand{\Res}{\mathop{\rm Res}\nolimits}

\newcommand{\sgn}{\mathop{\rm sgn}\nolimits}

\newcommand{\X}{\mathcal X}
\newcommand{\Y}{\mathcal Y}
\newcommand{\D}{\mathcal D}

\newcommand{\Yb}{\overline{Y}}
\newcommand{\Xb}{\overline{X}}

\newcommand{\nr}{{\scriptstyle \rm nr}}

\newcommand{\QQbar}{\overline{\mathbb{Q}}}
\newcommand{\kbar}{\overline{k}}
\newcommand{\Kbar}{\overline{K}}
\newcommand{\Ybar}{\overline{Y}}
\newcommand{\OKSu}{\mathcal{O}_{K,S_K}^*}

\def\phi{\varphi}

\newcommand{\iso}{\stackrel{\sim}{\rightarrow}}

\newcommand{\gen}[1]{\langle #1 \rangle}

\def\ie{i.e.}
\def\eg{e.g.}
\def\loccit{loc.\ cit.}

\usepackage{multirow}
\makeatletter
\renewcommand*\env@matrix[1][*\c@MaxMatrixCols c]{%
  \hskip -\arraycolsep
  \let\@ifnextchar\new@ifnextchar
  \array{#1}}
\makeatother

\hypersetup{
  pdfauthor   = {Bouw, Koutsianas, Sijsling, Wewers},
  pdftitle    = {Conductor and discriminant of Picard curves},
  pdfsubject  = {},
  pdfkeywords = {},
  backref=true, pagebackref=true, hyperindex=true, colorlinks=true,
  breaklinks=true, urlcolor=blue, linkcolor=blue, citecolor=blue,
  bookmarks=true, bookmarksopen=true}

\begin{document}

\title{Conductor and discriminant of Picard curves}
%\date{\today}

\begin{abstract}
  We describe normal forms and minimal models of Picard curves, discussing
  various arithmetic aspects of these. We determine all so-called special
  Picard curves over $\QQ$ with good reduction outside $2$ and $3$, and use
  this to determine the smallest possible conductor a special Picard curve may
  have.  We also collect a database of Picard curves over $\QQ$ of small
  conductor.
\end{abstract}

\author[Bouw]{Irene I.\ Bouw}
\address{%
  Irene I.\ Bouw,
  Institut für Reine Mathematik,
  Universität Ulm,
  Helmholtzstrasse 18,
  89081 Ulm,
  Germany
}
\email{irene.bouw@uni-ulm.de}

\author[Koutsianas]{Angelos Koutsianas}
\address{%
Angelos Koutsianas,
Department of Mathematics,
The University of British Columbia,
1984 Mathematics Road,
Vancouver, British Columbia,
V6T 1Z2,
Canada}
\email{akoutsianas@math.ubc.ca}

\author[Sijsling]{Jeroen Sijsling}
\address{%
  Jeroen Sijsling,
  Institut für Reine Mathematik,
  Universität Ulm,
  Helmholtzstrasse 18,
  89081 Ulm,
  Germany
}
\email{jeroen.sijsling@uni-ulm.de}

\author[Wewers]{Stefan Wewers}
\address{%
  Stefan Wewers,
  Institut für Reine Mathematik,
  Universität Ulm,
  Helmholtzstrasse 18,
  89081 Ulm,
  Germany
}
\email{stefan.wewers@uni-ulm.de}


\maketitle
%\tableofcontents

\section*{Introduction}

Let $K$ be a field whose characteristic does not equal $3$. A \emph{Picard
curve} $Y$ over $K$ is a smooth projective curve of genus $3$ over $K$ whose
base extension to the algebraic closure $\Kbar$ of $K$ admits a Galois morphism
of degree $3$ to $\PP^1$, see Definition \ref{def:Picard}. Over $\Kbar$, the
curve $Y$ can be described by an equation of the form
\begin{equation}\label{eq:intro1}
  Y :\; y^3 = f (x) = x^4 + c_1 x^3 + c_2 x^2 + c_3 x + c_4 \tag{$\ast$}
\end{equation}
with $f$ separable.  A curve $Y$ admitting an equation \eqref{eq:intro1} over
$K$ is called \emph{superelliptic of exponent $3$} over $K$.  Our definition is
slightly more general than the usual one from \cite{Holzapfel} in that we do
not assume that a Picard curve $Y/K$ admits a superelliptic equation of
exponent $3$ over $K$.

Picard curves are interesting because they furnish a family of superelliptic
curves that are nonhyperelliptic in the smallest genus, namely $3$, where such
a family exists. As such, they form a logical starting point for generalizing
notions developed for hyperelliptic and elliptic curves to superelliptic curves
of exponent greater than or equal to $3$. Like elliptic curves, Picard curves
are also smooth plane curves, as one sees by homogenizing the equation
\eqref{eq:intro1}.  Therefore one can use tools from invariant theory as, \eg,
in \cite{LLR} to study Picard curves. By contrast, hyperelliptic curves are
never complete intersections.  In this paper we explore the relations between
the superelliptic and planar vantage points, most importantly to study the
relation between the conductor and the minimal discriminant of a Picard curve
defined over a number field.

\bigskip
We start the paper by studying normal forms for Picard curves. Our approach
generalizes that of \cite[Appendix]{Holzapfel} but also works for residue
characteristic $2$. The first main result (Theorem \ref{thm:shortweq}) states
that every Picard curve over $K$ admits an equation \eqref{eq:intro1} over $K$,
\ie, is superelliptic of exponent $3$ over $K$, with one exception. We show
that the Picard curves that do not admit such an equation form a single
isomorphism class over $\Kbar$, which is characterized by the fact that its
automorphism group is as large as possible. We call these \emph{special Picard
curves}. Theorem \ref{thm:special_sweq} states that special Picard curves are
superelliptic of exponent $4$ over $K$, provided $\Char(K)\neq 2$.

In Section \ref{sec:integred} we recall the definition of the minimal
discriminant $\Delta^{\rm min}(Y)$ of a Picard curve $Y$ defined over a number
field $K$. We show that every nonspecial Picard curve admits a long Weierstrass
equation \eqref{eq:longweq} with minimal discriminant. Away from residue
characteristic $3$ there even exists a short Weierstrass equation
\eqref{eq:minshort} with minimal discriminant (Theorem \ref{thm:minshort}).
The analogous result for special Picard curves is Theorem
\ref{thm:specminshort}.

The primes dividing the minimal discriminant $\Delta^{\rm min}(Y)$ are exactly
the primes of bad reduction of $Y$  (Proposition \ref{prop:goodredternary}). We
obtain an alternative criterion for good reduction using the point of view of
covers of the projective line in Propositions \ref{prop:goodredbinary} and
\ref{prop:special_reduction}.  This criterion is formulated in terms of the
discriminant of the binary form $f$ from \eqref{eq:intro1}.

From Section \ref{sec:special} onward we concentrate on special Picard curves.
Our main result here is Theorem \ref{thm:mincond}, which states that the
smallest possible value for the conductor of a special Picard curve over $\QQ$
is $2^6 3^6$. This value is attained for the standard special Picard curve
$Y_{0}$ defined by \eqref{eq:special}. To prove this, we apply results from
\cite{superell} on the computation of the conductor of a superelliptic curves
via stable reduction. In Section \ref{sec:conductor1} we extend results from
\cite{Picard1} for nonspecial Picard curves to special ones, and prove lower
bounds on the local conductor $f_p$ at a prime of bad reduction. Our results
illustrate how to analyze the effect of twisting on the conductor.

The key ingredient in the proof of Theorem \ref{thm:mincond} is Theorem
\ref{thm:special_2_3}, which classifies all special Picard curves defined over
$\QQ$ with good reduction outside $2$ and $3$: there are precisely 800
different $\QQ$-isomorphism classes. The proof uses methods in Galois
cohomology that generalize beyond this particular case.

We expect $N=2^6 3^6$ to be the smallest conductor for any Picard curve defined
over $\QQ$. Following the strategy for studying this question outlined in
\cite[Section 5]{Picard1}, we have constructed a large database of Picard
curves over $\QQ$ that have good reduction outside two small primes, and more
precisely outside of the pairs $\left\{ 2, 3 \right\}$, $\left\{ 3 , 5
\right\}$, and $\left\{ 3, 7 \right\}$. These curves were obtained by methods
described in the forthcoming work \cite{kmr} (summarized in Appendix
\ref{sec:modified_Smart}) as well as an effective enumeration by Sutherland
\cite{Sutherland_database}. Our database gives equations for these Picard
curves, as well as their invariants, discriminants, and conductors. Its
construction is briefly discussed in the concluding Appendix
\ref{sec:database}.

The database provides evidence for the question, discussed in Section
\ref{sec:upperbound}, whether the conductor of a Picard curve divides its
minimal discriminant, as is the case for curves of genus $1$ and $2$.

\subsection*{Notations and conventions}

In this article, a curve is a separated scheme of dimension $1$ over a field.
Given an affine equation for a curve, we will identify it with the smooth
projective curve with the same function field.

\numberwithin{equation}{section}
\section{Picard curves}\label{sec:curves}

Let $K$ be a field of characteristic different from $3$. We fix an algebraic
closure $\Kbar$ of $K$.

\begin{definition}\label{def:Picard}
  A \emph{Picard curve} $Y$ over $K$ is a curve of genus $3$ over $K$ such that
  the base extension $Y_{\Kbar} = Y \otimes_K \Kbar$ admits a morphism
  $\phi_{\Kbar} : Y_{\Kbar} \to \PP^1_{\Kbar}$ that is a Galois cover of degree
  $3$.
\end{definition}

\begin{remark}
  This definition of a Picard curve is slightly more general than that in
  \cite[Appendix I]{Holzapfel} and \cite{Picard1}, in that we do not require
  the Galois cover $\phi_{\Kbar}$ to be defined over $K$.
\end{remark}

Given a Picard curve $Y$ over $K$, a morphism $\phi_{\Kbar}$ as in Definition
\ref{def:Picard} corresponds to a subgroup $G \subset \Aut_{\Kbar} (Y_{\Kbar})$
of order $3$ such that the quotient $Y_{\Kbar} / G$ has genus $0$.  We call the
group $G$ a \emph{distinguished subgroup of automorphisms} of $Y_{\Kbar}$.

\begin{lemma}\label{lem:shortweqbar}
  Suppose that $K = \Kbar$. Let $Y$ be a Picard curve over $K$, and let $G$ be
  a fixed distinguished subgroup of automorphisms of $Y$. Then there exist
  affine coordinates $x$ and $y$ on $Y$ that furnish an equation
  \begin{equation}\label{eq:shortweq}
    Y :\; y^3 = f (x) = x^4 + c_1 x^3 + c_2 x^2 + c_3 x + c_4
  \end{equation}
  for which $f$ is monic and separable of degree $4$ and for which $G$
  corresponds to the group of automorphisms generated by
  \begin{equation}
    \sigma : \quad y \mapsto \zeta_3 y, \quad x \mapsto x .
  \end{equation}
\end{lemma}

\begin{proof}
  By definition, $X = Y/G$ is a curve of genus $0$. Since $K$ is algebraically
  closed, we may identify $X$ with $\PP^1_K$ by choosing a coordinate $x$ on
  $X$.  Moreover, the Riemann--Hurwitz formula implies that the cover $\phi : Y
  \to \PP^1_K$ has exactly $5$ branch points. We choose the coordinate $x$ in
  such a way that $\infty$ is a branch point of $\phi$.

  By Kummer theory, $K(Y) = K (X) [y] / (y^3 - f)$, where $f \in K(X)$ is not a
  third power. Multiplying $f$ by a suitable third power, we may assume that
  $f$ is a monic polynomial of the form
  \begin{equation}
    f = \prod_{i=1}^4 (x - \alpha_i)^{e_i},
  \end{equation}
  with exactly $4$ pairwise distinct roots $\alpha_i$ and $e_i\in\{1,2\}$.
  Moreover, since $\phi$ branches at $\infty$, we have
  \begin{equation}
    e_1 + e_2 + e_3 + e_4 \not\equiv 0 \pmod{3}.
  \end{equation}
  Up to reordering the roots of $f$, there are precisely $4$ cases to consider:
  \begin{equation}
    (e_i)_i = (1, 1, 1, 1), \; (1, 1, 1, 2),
           \; (1, 2, 2, 2), \; (2, 2, 2, 2).
  \end{equation}
  Replacing $y$ by $y^{-1}$, if necessary, we may assume that $(e_i)_i = (1, 1,
  1, 1)$ or $(e_i)_i = (1, 1, 1, 2)$. If we are in the second case, then we can
  apply the change of coordinate $x \mapsto 1 / (x - \alpha_4)$ and obtain
  $(e_i)_i = (1, 1, 1, 1)$. In other words, we may assume that $f$ is a monic
  and separable polynomial of degree $4$. The lemma is now proved.
\end{proof}

\begin{remark}\label{rem:distpoint}
  As the proof of Lemma \ref{lem:shortweqbar} shows, the point at infinity of
  the model \eqref{eq:shortweq} is distinguished from the other branch points
  of $\phi$. Indeed, as a Galois cover with group $\ZZ / 3 \ZZ$, we have
  obtained identical local monodromy generators $1 \in \ZZ / 3 \ZZ$ at the
  finite branch points, so that the point at infinity has local monodromy
  generator $2 \in \ZZ / 3 \ZZ$.
\end{remark}

\begin{definition}
  We call \eqref{eq:shortweq} a \emph{short Weierstrass equation} for $Y$. For
  arithmetic reasons that will be described in Theorem \ref{thm:minshort}, we
  will also use this name for equations of the more general form
  \begin{equation}\label{eq:minshort}
    Y :\; b y^3 = c_0 x^4 + c_1 x^3 + c_2 x^2 + c_3 x + c_4
  \end{equation}
\end{definition}

\begin{corollary}\label{cor:nothyper}
  Let $Y$ be a Picard curve over $K$. Then $Y$ is not hyperelliptic.
\end{corollary}

\begin{proof}
  Since the polynomial $f$ in \eqref{eq:shortweq} is separable, the homogenized
  equation
  \begin{equation}\label{eq:shortweqproj}
    Y :\; F (y, x, z)
        = y^3 z - x^4 - c_1 x^3 z - c_2 x^2 z^2 - c_3 x z^3 - c_4 z^4
        = 0
  \end{equation}
  defines a smooth plane quartic model  for $Y$ in $\PP^2$. Now one concludes
  either by using the fact that hyperelliptic curves are not complete
  intersections or by noting that this embedding corresponds to the canonical
  morphism defined by the differentials $x \, \mathrm{d}x / y^2$, $y \,
  \mathrm{d}x / y^2$, $\mathrm{d}x / y^2$ of \eqref{eq:shortweq}.
\end{proof}

\begin{remark}
  We have ordered the variables in \eqref{eq:shortweqproj} as $y, x, z$, since
  this makes a few statements in this paper, particularly Lemma
  \ref{lem:normalize}, slightly more elegant.
\end{remark}

\begin{definition}\label{def:special}
  We call a Picard curve $Y$ over $K$ \emph{special} if its base extension
  $Y_{\Kbar}$ admits more than one distinguished group of automorphisms.
  Moreover, we call the curve
  \begin{equation} \label{eq:special}
    Y_0 :\; y^3 = x^4 - 1
  \end{equation}
  the \emph{standard} special Picard curve over $K$.
\end{definition}

\begin{lemma}\label{lem:special}
  Suppose that the characteristic of $K$ also does not equal $2$. Then
  a Picard curve over $K$ is special if and only if it is isomorphic over
  $\Kbar$ to the standard special Picard curve \eqref{eq:special}.
\end{lemma}

\begin{proof}
  This follows from the classification of the automorphism groups of plane
  curves in \cite[Theorem 3.1]{param}. For the exceptional case in \loccit,
  which is the Klein quartic in characteristic $7$, the statement of the lemma
  can be verified directly, for example by using that the relations satisfied
  by the Dixmier--Ohno invariants \cite{dixmier,elsenhans-invs,giko} of the
  curves \eqref{eq:shortweq} are not satisfied.
\end{proof}

Special Picard curves will be considered from Section \ref{sec:special} on;
until then, we will exclusively consider nonspecial Picard curves.

\begin{theorem}\label{thm:shortweq}
  Let $Y$ be a nonspecial Picard curve over a field $K$ with $\Char (K) \neq
  3$. Then $Y$ admits a short Weierstrass equation \eqref{eq:shortweq} defined
  by a monic polynomial $f$ over $K$.
\end{theorem}

\begin{proof}
  Since the distinguished group of automorphisms is defined over $K$, the
  ramification divisor of the morphism $\phi$ in the proof of Lemma
  \ref{lem:shortweqbar} is $K$-rational. By Remark \ref{rem:distpoint}, this
  divisor contains a distinguished point $P$.  Together, these statements imply
  that $X = Y / G$ is $K$-isomorphic to $\PP^1$. If we choose a coordinate $x$
  on $X$ for which $x (P) = \infty$, then $\phi$ ramifies over the divisor of a
  polynomial $f$ in $x$ with $K$-rational coefficients. This implies that $Y$
  is given by an equation
  \begin{equation}
    Y :\; y^3 = f(x) = c_0 x^4 + c_1 x^3 + c_2 x^2 + c_3 x + c_4 .
  \end{equation}
  where $f \in K[x]$ is separable of degree $4$. A change of coordinates $(x,
  y) \mapsto (c_0^{-1} x, c_0^{-1} y)$ ensures that $f$ is monic.
\end{proof}

\begin{remark}
  Theorem \ref{thm:shortweq} can also be proved by the same methods as Theorem
  \ref{thm:special_sweq}.
\end{remark}

The isomorphisms between nonspecial Picard curves admit a simple description,
as in the case of hyperelliptic curves.

\begin{lemma}\label{lem:isos}
  Let $Y_i :\; y^3 = f_i (x)$ be two nonspecial Picard curves over $K$.
  \begin{enumerate}[(a)]
    \item Let $\phi : Y_1 \to Y_2$ be an isomorphism. Then there exist $\alpha,
      \beta, \gamma \in K$, with $\alpha, \gamma \neq 0$, such that
      \begin{equation}\label{eq:isos1}
        \phi^* (x) = \alpha x + \beta, \quad \phi^* (y) = \gamma y .
      \end{equation}
    \item There exists an isomorphism $\phi : Y_1 \to Y_2$ if and
      only if there exist $\alpha, \beta, \gamma \in K$ with $\alpha, \gamma
      \neq 0$ such that
      \begin{equation}\label{eq:isos2}
        f_2 (\alpha x + \beta) = \gamma^3 f_1 (x),
      \end{equation}
      in which case we can take $\phi (x, y) = (\alpha x + \beta,
      \gamma y)$.
  \end{enumerate}
\end{lemma}

\begin{proof}
  Let $G_i$ denote the distinguished automorphism group of the curve
  $Y_{i,\bar{K}}$. Since $Y_1, Y_2$ are assumed to be nonspecial, we have $G_2
  = \phi G_1 \phi^{-1}$, implying that $\phi$ induces an isomorphism $\psi :
  X_1 \to X_2$ between the quotients $X_i = Y_i / G_i$.

  The morphism $\phi$ maps the ramification (resp.\ branch) divisor of $Y_1 \to
  X_1$ onto the ramification (resp.\ branch) divisor of $Y_2 \to X_2$. We have
  identified the curves $X_i$ with the projective line $\PP^1_K$ as in the
  proof of Theorem \ref{thm:shortweq}. Therefore $\psi$ fixes the distinguished
  point $\infty$ in the proof of Theorem \ref{thm:shortweq}.  We get
  \begin{equation}
    \phi^* (x) = \psi^* (x) = \alpha x + \beta,
  \end{equation}
  with $\alpha, \beta \in K$ and $\alpha \neq 0$.

  Moreover, $\phi$ maps the divisors of zeros and poles of $y$ to one another.
  We conclude that $\phi^* (y) = \gamma y$ for a constant $\gamma\in K^\times$.
  Hence
  \begin{equation}
    \phi^* (f_2) = \phi^* (y^3) = \phi^* (y)^3 = \gamma^3 y^3 = \gamma^3 f_1 .
  \end{equation}
  We have shown (a) and the forward direction of (b). The reverse direction of
  (b) can be verified by a calculation.
\end{proof}

\begin{remark}
  Alternatively, one notes that any isomorphism between the curves $Y_1$ and
  $Y_2$ is induced by an element of the normalizer of the common automorphism
  group of \eqref{eq:shortweqproj}, after which one applies \cite[Theorem
  A.1.(iv)]{param}.
\end{remark}

\numberwithin{equation}{subsection}
\section{Nonspecial Picard curves as plane curves}\label{sec:integred}

In this section we consider integral equations and reduction properties of a
nonspecial Picard curve $Y$. To this end, we fix a discrete valuation $v$ on
our field $K$, with uniformizer $\pi$, valuation ring $\OO_K$ and residue field
$k$. In Section \ref{sec:disc}, we review the \emph{discriminant} of ternary
quartic forms, which gives a way to quantify bad reduction behavior of an
integral quartic model of $Y$. In Section \ref{sec:mindisc}, we study the
integral models of $Y$ that give rise to the smallest possible discriminant. We
show that if $\Char (k) \neq 3$, we can find such a minimal model that is
defined by a short Weierstrass equation. When $\Char (k) = 3$, we typically
need a more general integral equation called the \emph{long Weierstrass
equation}, as defined in \eqref{eq:longweq}, to attain a model of minimal
discriminant valuation.

While we limit ourselves to local considerations, our results apply more
globally over number rings with trivial class group, as we will occasionally
clarify in a remark.

\subsection{The discriminant}\label{sec:disc}

We refer to \cite{Demazure} and \cite[Chapter 13]{GKZ}, for a more complete
exposition of what follows.

Let $F \in R [y, x, z]$ be a ternary quartic form over a domain $R$ whose
fraction field does not have characteristic $2$. The \emph{discriminant}
$\Delta (F) = \Delta_{3,4} (F)$ of $F$ is the element of $R$ defined as
\begin{equation}\label{eq:disc_plane}
  \Delta (F) = \Res(\mathrm{D}_y F, \mathrm{D}_x F, \mathrm{D}_z F) / 2^{14} .
\end{equation}
Here $\mathrm{D}_T F$ denotes the partial derivative of $F$ with respect to the
variable $T$, and $\Res$ denotes the resultant of $3$ ternary quartics.

\begin{remark}
  In \cite{LLR}, it is shown that the discriminant $\Delta (F)$ is related to
  the Dixmier--Ohno invariant $I_{27} (F)$ of $F$ via $\Delta (F) = 2^{40}
  I_{27}(F)$.
\end{remark}

\begin{proposition}\label{prop:discprop}
  Let $F \in K [y, x, z]$. The discriminant satisfies the following properties.
  \begin{enumerate}[(a)]
    \item $F$ defines a nonsingular curve over $K$ if and only if $\Delta (F)
      \in K^{\times}$.
    \item If $F \in \OO_K [y, x, z]$ is integral, then the reduction of $F$
      modulo $v$ defines a nonsingular curve over $k$ if and only if $v (\Delta
      (F)) = 0$.
    \item $\Delta$ is a homogeneous form of degree $27$ in the coefficients of
      $F$, so that $v (\Delta (\lambda F)) = v (\Delta (F)) + 27 \lambda$.
    \item $\Delta$ has weight $36$, \ie, $v (\Delta (F \circ T)) = v
      (\Delta (F)) + 36 v (\det (T))$ for $T \in \GL_3 (K)$.
  \end{enumerate}
\end{proposition}

\begin{proof}
  The first two statements follow from the properties of $\Delta$ described in
  \cite[Section 4]{Demazure}, whereas the third results from $\Delta$ being
  homogeneous of degree $27$ in the coefficients of $F$.  The final statement
  follows from the fact that an invariant of ternary quartics of degree $3 k$
  has weight $4 k$, as recalled in for example
  \cite[(1.11)]{LRS-Reconstruction}.
\end{proof}

The following statement follows from the properties of the resultant in
\cite[Section 4]{Demazure}. It will be used in many constructions and examples
to follow, like Remark \ref{rem:red_min}.

\begin{lemma}\label{lem:disc_nonspecial}
  Let $f = c_0 x^4 + c_1 x^3 + c_2 x^2 + c_3 x + c_4$ be a separable polynomial
  over $K$ with discriminant $\Delta (f)$, and let $F (y, x, z)$ be a
  corresponding short Weierstrass equation \eqref{eq:minshort}. Then
  \begin{equation}
    \Delta(F) =  -3^9 b^{12} c_0^3 \Delta(f)^2 .
  \end{equation}
\end{lemma}

\begin{definition}\label{def:mindisc}
  Let $Y$ be a Picard curve over a discretely valued field $(K, v)$. The
  \emph{minimal discriminant exponent} of $Y$ at $v$ is the integer
  \begin{equation}
    e_v (Y) = \min_F \left\{ v (\Delta (F)) \right\} \in \ZZ_{\ge 0} .
  \end{equation}
  Here $F$ runs through the quartic forms $F \in \OO_K [x,y,z]$ that furnish a
  model of $Y$. The \emph{minimal discriminant} of $Y$ at $v$ is the ideal
  \begin{equation}
    (\Delta^{\min} (Y)) = (\pi^{e_v (Y)}) .
  \end{equation}

  For a Picard curve $Y$ over a number field $K$, we define the \emph{minimal
  discriminant} of $Y$ to be the integral ideal
  \begin{equation}\label{eq:mindisc}
    \Delta^{\min} (Y) = \prod_{\p} \p^{e_{v_{\p}} (Y)} .
  \end{equation}
  Here $\p$ runs through the primes of $K$ and  $v_{\p}$ is the valuation
  corresponding to $\p$.
\end{definition}

\begin{remark}
  In what follows, we will identify the minimal discriminant of a Picard curve
  over $\QQ$ with the positive generator of the integral ideal
  \eqref{eq:mindisc}.
\end{remark}

\begin{remark}
 The minimal discriminant was also discussed in \cite[Section
   4.2]{Picard1}. We take this opportunity to correct a mistake.
In \cite{Picard1}, above Lemma 4.3,  it is claimed that one can find an
  equation $y^3 = f(x)$ for any Picard curve such that $v_p (\Delta (f)) < 36$.
  As Elisa Lorenzo García pointed out to us, this is not the case. A
  counterexample is the family of curves
  \begin{equation}
    y^3 = f_n(x) = x^4 + x^2 + p^n\in \QQ_p[x],
  \end{equation}
  whose discriminant valuation at $p$ gets arbitrarily large as $n$ goes to
  $\infty$.  However, the discriminant of the
  Weierstrass equation, and hence of $f$, is minimal for all $n$.
%
%  \bigskip
%  Second, if $\Char(k) = 2$, the condition $v_2 (\Delta (f)) < 36$ does not
%  guarantee that the equation $y^3 = f(x)$ is minimal. Indeed, let $Y/\QQ$ be
%  the Picard curve with defining equation $y^3 = f(x) =  x^4 + x^3 + x^2 + x +
%  1$. Applying the coordinate
%  transformation $(x, y) \mapsto ((x - 1) / 4, y/8)$ we obtain an alternative
%  model
%  \begin{equation}\label{eq:nonred2}
%    y^3 = 2 g (x), \quad \text{ where } \quad
%    g(x) = x^4 + 10 x^2 + 40 x + 205,
%  \end{equation}
%  with $v_2 (\Delta (g)) = 24 < 36$. However, $g$ and hence the model
%  \eqref{eq:nonred2} is not reduced at $p = 2$.
\end{remark}

\subsection{Discriminant minimization}\label{sec:mindisc}

In this section, we  derive a standard model for Picard curves over the
discretely valued field $K$. To this end, we first prove a lemma on smooth
plane curves.

\begin{lemma}\label{lem:normalize}
  Let $X = V (F) \subset \PP^2_K$ be a smooth plane curve over $K$ defined by a
  ternary quartic form $F \in \OO_K [x, y, z]$. Let $P \in X (K)$ be a rational
  point on $X$.
  \begin{enumerate}[(a)]
    \item There exists a matrix $T_1 \in \GL_3 (\OO_K)$ that maps $P$ to $(1 : 0
      : 0)$.
    \item Suppose that $P = (1 : 0 : 0)$. Then there exists a matrix
      $T_2 \in \GL_3 (\OO_K)$ that fixes $P$ and that maps the tangent line $L$
       of $X$ in $P$ to $z = 0$.
    \item Let $X_1$ and $X_2$ be two plane curves with the common rational
      point $P_1 = P_2 = (1 : 0 : 0)$ and the common tangent line $z = 0$ at
      these points. Then any isomorphism between $X_1$ and $X_2$ sending $P_1$
      to $P_2$ is induced by an upper triangular matrix $T$ in $\GL_3 (K)$.
  \end{enumerate}
\end{lemma}

\begin{proof}
  (a) It suffices to construct a matrix $U_1 \in \GL_3 (\OO_K)$ sending $(1 : 0 :
  0)$ to $P$ instead. Suppose, without loss of generality, that the first
  coordinate of $P$ has trivial valuation. Then we can take $U_1$ to be the
  matrix whose first column is given by the coordinates of $P$ and whose other
  columns are the standard basis vectors.

  (b) The tangent line $L$ is of the form $\gamma x + \delta z = 0$. We may
  choose $\gamma$ and $\delta$ to be integral and coprime. Choosing $\alpha$
  and $\beta$ integral such that $\alpha \gamma + \beta \delta = 1$, we can
  take $T_2$ to be the inverse of the matrix corresponding to the
  transformation $y \mapsto y$, $x \to \alpha x + \beta z$, $z \mapsto \gamma x
  + \delta z$.

  (c) Any isomorphism between smooth plane curves is induced by an invertible
  matrix $T = (T_{i,j}) \in \GL_3 (K)$. Since $T$ has to fix $(1 : 0 : 0)$, we
  have $T_{2,1} = T_{3,1} = 0$. And since $T$ fixes the common tangent line $z
  = 0$, we need a corresponding condition on its transpose, yielding $T_{3,1} =
  T_{3,2} = 0$.
\end{proof}

\begin{remark}\label{rem:global}
  The same considerations apply globally over a number field with trivial class
  group. For example, the proof of part (a) of the proposition in this more
  global case uses the fact that given a PID $R$, an inclusion $0 \to R \to
  R^3$ with torsion-free quotient $Q$ always admits a splitting $Q \to R^3$.
  Because of this, we can always augment a coprime coordinate vector for $P$ to
  an invertible matrix over $R$, after which the same argument can be run.
\end{remark}

\begin{lemma}\label{lem:minlong}
  Let $Y$ be a nonspecial Picard curve over $K$, and let $F \in K [y, x, z]$
  be a ternary form defining a curve isomorphic to $Y$. Then after an integral
  transformation in $\GL_3 (\OO_K)$, the ternary quartic $F$ yields an equation for
  $Y$ of the form
  \begin{equation}\label{eq:longweq}
    Y :\; (a_0 y^3 + a_1 (x, z) y^2 + a_2 (x, z) y) z = a_4 (x, z) ,
  \end{equation}
  where $a_i$ is a homogeneous form of degree $i$ in $x$ and $z$, and where
  \begin{equation}\label{eq:longweq_rel}
    a_1^2 = 3 a_0 a_2 .
  \end{equation}
  Conversely, any equation \eqref{eq:longweq} satisfying \eqref{eq:longweq_rel}
  defines an equation \eqref{eq:shortweqproj} of a Picard curve via the
  substitution $y \mapsto y - (a_1 (x, z)/ (3 a_0))$.
\end{lemma}

\begin{proof}
  The nonspecial Picard curve $Y$ has a distinguished rational point $P$,
  corresponding to the point $(1 : 0 : 0)$ in \eqref{eq:shortweqproj}. By (a)
  and (b) in Lemma \ref{lem:normalize}, we may apply an integral transformation
  to suppose that the point on the curve defined by $F$ corresponding to $P$ is
  again given by $(1 : 0 : 0)$, with tangent line $z = 0$.

  Part (c) of the same lemma then shows that $F$ can be obtained from a short
  Weierstrass equation by a substitution $y \mapsto \lambda y + \mu (x, z)$.
  This yields an equation of the form \eqref{eq:longweq}. The relation
  \eqref{eq:longweq_rel} is satisfied because $a_0 = \lambda^3, a_1 = 3
  \lambda^2 \mu, a_2 = 3 \lambda \mu^2$. The converse statement can be checked
  directly.
\end{proof}

\begin{definition}
  Given a Picard curve $Y$, we call an equation for $Y$ of the form
  \eqref{eq:longweq} a \emph{long Weierstrass equation} for $Y$.
\end{definition}

\begin{proposition}\label{prop:minlong}
  Let $Y$ be a nonspecial Picard curve over the discretely valued field $(K,
  v)$. Then $Y$ admits an integral long Weierstrass equation of minimal
  discriminant exponent.
\end{proposition}

\begin{proof}
  One follows the proof of Lemma \ref{lem:minlong}. By Proposition
  \ref{prop:discprop}.(d), the integral transformation in the proof of this
  lemma preserves the minimality of the discriminant of $F$.
\end{proof}

\begin{remark}\label{rem:elsenhans}
  An equation with minimal discriminant for a given Picard curve $Y$ over $\QQ$
  can often be found explicitly by using the algorithms in \cite{elsenhans}
  based on \cite{Kollar}. Our methods then yield a long Weierstrass equation
  for $Y$. Occasionally, though, Elsenhans' algorithms will only return
  equations that are locally minimal in the Bruhat--Tits tree corresponding to
  the integral models of $Y$.
\end{remark}

\begin{remark}
  Integral long Weierstrass equations with the same discriminant exponent need
  \emph{not} be related by an integral transformation, essentially because one
  may be able to divide out scalars after substitution. For example, for $p
  \neq 2$ the equation
  \begin{equation}
    Y_1 :\; p y^3 z = p x^4 + 2 x^2 z^2 + p^3 x z^3 - p^4 z^4
  \end{equation}
  can be transformed into
  \begin{equation}
    Y_2 :\; y^3 z = p^5 x^4 + 2 x^2 z^2 + p x z^3 - z^4
  \end{equation}
  by substituting $p y$ for $y$ and $p^2 x$ for $x$. This implies that none of
  the isomorphisms $Y_1 \to Y_2$ is defined by an element of $\GL_3 (\OO_K)$. Yet
  both of these equations have discriminant valuation $25$ at $p$.
\end{remark}

\begin{theorem}\label{thm:minshort}
  Let $Y$ be a nonspecial Picard curve over a discretely valued field $(K, v)$
  whose residue characteristic does not equal $3$. Then $Y$ admits an integral
  short Weierstrass equation \eqref{eq:minshort} of minimal discriminant
  exponent.
\end{theorem}

\begin{proof}
  Extend $v$ to the Gauss valuation for ternary polynomials, and consider an
  integral equation \eqref{eq:longweq} of $Y$ whose discriminant is minimal.
  First suppose that $v (a_0) = 0$. Then the integral coordinate change $y
  \mapsto y - (a_1 (x, z)/ (3 a_0))$ gives an equation of the form
  \eqref{eq:minshort}, and we are done. If $v (a_0) = 1$, then $v (a_1) \geq 1$
  by \eqref{eq:longweq_rel}, and the same argument applies.

  Now suppose that $v (a_0) \ge 2$. Then $v (a_1) \ge 1$ by
  \eqref{eq:longweq_rel}. In this case substituting $y \mapsto \pi y$ and $z
  \mapsto \pi z$ yields a ternary form with Gauss valuation at least $3$. Since
  the discriminant has degree $27$ and weight $36$, dividing out the common
  factor of the coefficients yields a new integral equation whose discriminant
  exponent is at least $3 \cdot 27 - 2 \cdot 36$ smaller. This is in
  contradiction with our minimality assumption.
\end{proof}

\begin{remark}
  It is not true that if we suppose additionally that the residue
  characteristic of $\OO_K$ does not equal $2$, then $Y$ will admit an integral
  short Weierstrass equation of the more restricted form
  \begin{equation}\label{eq:traceless}
    b y^3 = c_0 x^4 + c_2 x^2 + c_3 x + c_4
  \end{equation}
  whose discriminant exponent is minimal. Note how this contrasts with the
  behavior of elliptic curves, which always admit a minimal Weierstrass
  equation $y^2 = x^3 + a x + b$ away from $2$ and $3$.

  For a counterexample, consider the curve
  \begin{equation}\label{eq:nonminshort}
    Y_1 : y^3 z = p x^4 + x^3 z + 2 x^2 z^2 + x z^3 - z^4
  \end{equation}
  for $p = 17$ (or large enough).  The exponent of the discriminant of
  \eqref{eq:nonminshort} at $p$ equals $3$, which then has to be the minimal
  discriminant exponent at $p$ because of (c) and (d) of Proposition
  \ref{prop:discprop}.  A nonintegral equation of the form \eqref{eq:traceless}
  is obtained by substituting $x\mapsto x - (z / 4 p)$, which yields
  \begin{equation}
    Y_2 :\; y^3 z = p x^4 + \frac{269}{2^3 p} x^2 z^2
                  - \frac{2177}{2^3 p^2} x z^3
                  + \frac{1275683}{2^8 p^3} z^4
  \end{equation}
  Because of Lemma \ref{lem:isos}.(b), all isomorphisms between curves defined
  by equations of the form \eqref{eq:traceless} are induced by diagonal
  matrices. This makes finding an integral equation of the form
  \eqref{eq:traceless} with minimal discriminant (among all equations of that
  particular form) into a problem in linear programming. In this particular
  case, the minimal discriminant valuation at $p$ is attained by
  \begin{equation}
    Y_3 :\; y^3 z = x^4 - \frac{269}{2^3} x^2 z^2
                  - \frac{2177}{2^3} x z^3 + \frac{1275683}{2^8} z^4.
  \end{equation}
  It equals $12$, which is strictly larger than the valuation $3$ obtained by
  using the more general class of equations \eqref{eq:minshort}.
\end{remark}

\subsection{Criteria for good reduction}\label{sec:badred}

Let $(K, v)$ be a discretely valued field with $\Char (K) \neq 3$.  In this
section we only deal with questions local to $v$, and therefore can and will
assume that $K$ is complete with respect to $v$.

\begin{definition}\label{def:goodred}
  Let $Y$ be a Picard curve over $K$. We say that $Y$ has \emph{good reduction}
  at $v$ if there exists a smooth and proper $\OO_K$-model $\Y \to \Spec \OO_K$
  of $Y$.  If not, we say that $Y$ has \emph{bad reduction}. We say that $Y$
  has \emph{potentially good reduction} if there exists an extension $L/K$ such
  that $Y_L = Y\otimes_K L$ has good reduction.
\end{definition}

\begin{proposition}\label{prop:goodredternary}
  A Picard curve $Y$ over $K$ has good reduction if and only if $v
  (\Delta^{\min} (Y)) = 0$.
\end{proposition}

\begin{proof}
  First assume that $v (\Delta^{\min} (Y))= 0$. Let $F \in \OO_K [y,x,z]$ be a
  quartic form with $v (\Delta (F)) = 0$. Then $F$ defines a model $\Y_F$ of
  $Y$ over $\OO_K$. Proposition \ref{prop:discprop}  shows that $\Y_F$ is
  smooth.

  Now assume that $Y$ has good reduction. Let $\Y$ be a smooth model of $Y$
  with smooth special fiber $\Ybar$. We show that $\Ybar$ is nonhyperelliptic.
  Since this property of $\Ybar$, as well as its negation, is stable under base
  extension, we may assume that $k = \kbar$ is algebraically closed. Then
  \cite[Proposition 10.3.38]{liu} shows that $\Aut_{\Kbar} (Y_{\Kbar})$ can be
  considered as a subgroup of $\Aut_{\kbar} (\Ybar_{\kbar})$. More precisely,
  $\Aut_{\kbar} (\Ybar_{\kbar})$ contains a cyclic subgroup $G$ of order $3$
  such that $g(\overline{Y} / G) = 0$.

  If $\Char (k) \neq 3$, then we conclude that $\Ybar$ is a Picard curve, so
  that it is nonhyperelliptic by Corollary \ref{cor:nothyper}. If $\Char (k) =
  3$, then the Riemann--Hurwitz formula implies that there are two
  possibilities for the cover $\pi : \overline{Y} \to \overline{Y} / G \simeq
  \PP^1$. In the former case $\pi$ has a unique branch point.
  Artin--Schreier theory implies that $\overline{Y}$ admits an equation $y^3 - y
  = x^4$ over $\kbar$. In the latter case $\pi$ has two branch points and over
  $\kbar$ the curve $\overline{Y}$ admits an equation of the form $x y^3 - x y
  = g (x)$ for some polynomial $g$ of degree $3$. So also in this case
  $\overline{Y}$ is nonhyperelliptic.

  As at the beginning of \cite[\S 2.1]{LLR}, we now use the relative canonical
  sheaf of $\Y$ to obtain a proper $\OO_K$-morphism $\phi : \Y \to \PP^2$. If
  $\phi$ is not a smooth embedding, then \loccit\ shows that its special fiber
  is a degree $2$ map to a conic over $k$, so that $\Ybar$ is hyperelliptic. We
  have just excluded this possibility, so that $\phi$ induces an isomorphism of
  $\Y$ with its image in $\PP^2$. By dimension theory, this image is defined by
  a single ternary quartic form $F \in \OO_K [y,x,z]$. Since both $F$ and its
  reduction define a smooth curve, we can once more invoke Proposition
  \ref{prop:discprop} to conclude that $v (\Delta (F)) = 0$ and therefore $v
  (\Delta^{\min} (Y)) = 0$.
\end{proof}

\begin{remark}
  Contrary to Picard curves, a general plane quartic curve over $K$ with good
  reduction may not admit a smooth plane quartic model over $\OO_K$, as it may
  have good reduction that is hyperelliptic. Necessary and sufficient
  conditions for this to happen are given in \cite{LLR}.
\end{remark}

Our results also give a criterion for bad reduction that gives an alternative
viewpoint of the result in \cite[Proposition 3.4]{Picard1}.

\begin{proposition}\label{prop:badredwild}
  Let $Y$ be a nonspecial Picard curve over a discretely valued field $(K, v)$
  for which $3$ has odd valuation. Let $Y$ be a nonspecial Picard curve $K$.
  Then $Y$ does not have good reduction.
\end{proposition}

\begin{proof}
  Suppose that $Y$ has good reduction. As in the proof of Proposition
  \ref{prop:goodredternary}, we obtain an integral equation
  \eqref{eq:longweq} whose discriminant valuation equals $0$ and whose
  reduction modulo $3$ is nonsingular. We have $v (a_0) = 0$ by nonsingularity
  of the reduction. The equation \eqref{eq:longweq_rel} then implies that $v
  (a_1) > 0$, and because $3$ has odd valuation, this in turn implies that $v
  (a_2) > 0$. This implies that the reduction is defined by an equation of the
  form $y^3 = f (x)$, which is a contradiction since such an equation does not
  define a smooth curve in characteristic $3$.
\end{proof}

\section{Nonspecial Picard curves as superelliptic curves}\label{sec:goodred}

Instead of minimizing the ternary form $F$ defining the Picard curve in
\eqref{eq:shortweqproj}, we can also  reduce the binary form $f$ figuring in
\eqref{eq:shortweq}. Indeed, Theorem \ref{thm:minshort} states the existence of
an integral short Weierstrass equation with minimal discriminant, but does not
provide an easy-to-use criterion for recognizing such minimal equations.

In this section we follow an alternative approach that modifies the binary form
$f$ from \eqref{eq:shortweqproj} rather than the whole equation. While this
approach does not necessarily find an equation with minimal discriminant, the
algorithm is far simpler, and suffices for the purpose of recognizing Picard
curves with good reduction, see Section \ref{sec:goodred2}. Moreover, an
extension of this approach can be used to calculate the stable reduction in the
case of bad reduction. This approach works in principle for all superelliptic
curves, and does not use that the curve is a complete intersection. We refer to
\cite{superell} for more details.

\subsection{Reducing binary quartics}\label{sec:binred}

Let $K$ be a field  and let $Y_1$ and $Y_2$ be two nonspecial Picard curves
over $K$. Lemma \ref{lem:isos} implies that isomorphisms over $K$ between $Y_1$
and $Y_2$ induce isomorphisms between the branch divisors $\hat{D_i}$ of the
degree-$3$ quotient maps $\phi: Y_1 \to X_1 = \PP^1_x$. One of the branch
points is distinguished, see Remark \ref{rem:distpoint}.  We call $X$ together
with a $(4,1)$-divisor a \emph{$(4,1)$-marked projective line} over $K$.

In what follows we choose a coordinate on $X$ such that the distinguished point
is $x=\infty$. Write $D$ for the finite part of the divisor $\hat{D}$. Then
there exists a unique monic quartic $f\in K [x]$  such that $D=(f)_0$ is the
divisor of zeros of $f$. To study these, we make the following definitions.

\begin{definition}
  \begin{enumerate}[(a)]
    \item A \emph{quartic} over $K$ is a monic and separable quartic polynomial
      $f = x^4 + c_1 x^3 + c_2 x^2 + c_3 x + c_4 \in K[x].$
    \item Two quartics $f_1, f_2\in K[x]$ are \emph{equivalent} if there exists
      a matrix
      \begin{equation}\label{eq:A}
        A = \begin{pmatrix}
              \alpha & \beta \\ 0 & 1
            \end{pmatrix}.
      \end{equation}
      such that
      \begin{equation}\label{eq:quartic_equiv}
        f_2=f_1 \circ A  := \alpha^{-4} f_1(\alpha x + \beta) .
      \end{equation}
  \end{enumerate}
\end{definition}

The factor $\alpha^{-4}$ is inserted in \eqref{eq:quartic_equiv} to ensure that
$f_1 \circ A$ is again monic.

Now again let $v$ be a discrete valuation on $K$ with valuation ring $\OO_K$
and residue field $k$.

\begin{definition}\label{def:binred}
  A quartic $f \in K[x]$ is called \emph{reduced} with respect to $v$ if the
  following holds.
  \begin{enumerate}[(a)]
    \item $f \in \OO_K[x]$.
    \item If $g \in \OO_K[x]$ is equivalent to $f$, then $v (\Delta(f)) \leq
      v (\Delta(g))$.
  \end{enumerate}
\end{definition}

For a quartic $f \in K[x]$ we shall write $-\lambda(f)\in\QQ$ for the largest
slope of the Newton polygon of $f$ with respect to $v$. In other words,
\begin{equation}
  \lambda (x^4 + c_1 x^3 + c_2 x^2 + c_3 x + c_4)
  = \min_{1 \leq i \leq 4} \frac{v (c_i)}{i}.
\end{equation}
We have that $\lambda(f)$ is the minimum of the valuations of the roots of $f$
in $\Kbar$, and that $f \in \OO_K[x]$ if and only if $\lambda(f) \geq 0$.

\begin{lemma}\label{lem:reduced_quartic}
  A quartic $f \in K[x]$ is reduced if both of the following conditions hold.
  \begin{enumerate}[(a)]
  \item $0 \leq \lambda(f) < 1$.
  \item If $\lambda(f) = 0$, then the image of $f$ in $k [x]$ is not a $4$th
    power.
  \end{enumerate}
\end{lemma}

\begin{proof}
  We assume that $f$ is not reduced. Then there exist $\alpha, \beta \in K$,
  $\alpha \neq 0$, such that
  \begin{equation}
    g = \alpha^{-4} f (\alpha x + \beta) \in \OO_K[x], \quad
    v (\Delta (g)) = v (\Delta (f)) - 12 v (\alpha) < v (\Delta (f)).
  \end{equation}
  From this we obtain that $v (\alpha) \geq 1$ and $v (\beta) \geq 0$. Let
  $L/K$ be the splitting field of $f$ and let $w$ be an extension of $v$ to
  $L$. Write
  \begin{equation}
    f = \prod_{i=1}^4 (x - \xi_i).
  \end{equation}
  Then
  \begin{equation}
    g = \prod_{i=1}^4 \left(x - \frac{\xi_i - \beta}{\alpha}\right).
  \end{equation}
  Since $f$ and $g$ are both integral, we have
  \begin{equation} \label{eq:reduced_quartic3}
    w (\xi_i) \geq 0, \quad
    w \left( \frac{\xi_i - \beta}{\alpha} \right)
    = w (\xi_i - \beta) - v (\alpha) \geq 0,
  \end{equation}
  for all $i$. Using $v(\alpha) > 0$, we conclude that for all $i$ we also have
  \begin{equation} \label{eq:reduced_quartic4}
    w (\xi_i - \beta) > 0.
  \end{equation}

  We have to show that either (a) or (b) is false. Assume that (a) is true,
  \ie,
  \begin{equation}
    \lambda (f) = \min_i w (\xi_i) < 1.
  \end{equation}
  Then there exists some $j \in \{ 1, \ldots, 4\}$ such that $w (\xi_j) < 1$.
  Since $w (\beta) \geq 0$ and $w(\beta) = v(\beta) \in \ZZ$, the strict
  triangle inequality, applied to \eqref{eq:reduced_quartic4}, shows that $w
  (\xi_j) = w (\beta) = 0$. But then \eqref{eq:reduced_quartic4} also implies
  that $w (\xi_i) = 0$ for all $i$. We conclude that
  \begin{equation}
    \lambda (f)  = 0 \quad \text{ and } \quad
    \overline{f} = (x - \overline{\beta})^4.
  \end{equation}
  So (b) is false, and the lemma is proved.
\end{proof}

\begin{remark}
  Conditions (a) and (b) are sufficient but not necessary for $f$ to be
  reduced. For example, the quartic $f = (x + 1)^4 + p$ over $\QQ$ is reduced
  with respect to the $p$-adic valuation $v_p$, but (b) does not hold.
\end{remark}

Algorithmically, a given quartic can be reduced as follows.

\begin{algorithm}\label{alg:red}
  Let $f$ be a quartic polynomial over $K$. The following algorithm returns a
  reduced quartic equivalent to $f$.
  \begin{enumerate}[(a)]
    \item Set $n = \lfloor \lambda (f) \rfloor$ and replace $f$ by $\pi^{-4n} f
      (\pi^n x)$. Now $0 \leq \lambda(f) < 1$.
    \item Suppose $\lambda (f)=0$. Let $\overline{f} \in k[x]$ denote the image of
      $f$. If $\overline{f} = (x - \overline{a})^4$, with $\overline{a}\in
      k^\times$, replace $f$ by $f(x + a)$, for some lift $a\in \OO_K$ of
      $\overline{a}$, and go back to (a).
    \item Now $f$ is reduced, by Lemma \ref{lem:reduced_quartic}.
  \end{enumerate}
\end{algorithm}

Note that Algorithm \ref{alg:red} terminates after finitely many rounds because
$v (\Delta (f))$ gets strictly smaller each time we repeat Step (b).

The next proposition relates reduced quartics to good reduction.

\begin{definition}\label{def:markedPP1}
  We say that a $(4, 1)$-marked projective line $(X, \infty, D)$ over $K$ has
  \emph{good reduction} with respect to $v$ if there exists a smooth and proper
  $\OO_K$-model $\X \to \Spec \OO_K$ of $X$ such that the scheme-theoretic
  closure $\hat{\D} \subset \X$ of the divisor $\hat{D} = D \cup \{ \infty \}$
  is étale over $\Spec \OO_K$.
\end{definition}

\begin{proposition}\label{prop:goodredbq}
  Let $f \in K[x]$ be a quartic and $(X, \infty, D)$ the corresponding $(4,
  1)$-marked projective line. Then the following are equivalent:
  \begin{enumerate}[(a)]
    \item $(X, \infty, D)$ has good reduction with respect to $v$.
    \item The discriminant of the reduced quartic equivalent to $f$ is
      a unit in $\OO_K$.
  \end{enumerate}
\end{proposition}

\begin{proof}
  Let us assume that $(X, \infty, D)$ has good reduction with respect to $v$,
  and let $\X$ be an $\OO_K$-model as in Definition \ref{def:markedPP1}. Then
  there exists an $\OO_K$-isomorphism $\X \cong \PP^1_{\OO_K}$ which sends
  $\infty$ to $\infty$. Let $g$ be the quartic (equivalent to $f$) whose
  divisor of zeros is the image of $D$ in $\PP^1_K$ under this isomorphism.
  Then $g \in \OO_K[x]$ (because the closure of $\infty$ in $\X$ is disjoint
  from the closure of $D$) and the image $\overline{g} \in k[x]$ of $g$ is
  separable (because the closure of $D$ is étale over $\Spec \OO_K$). It
  follows that $\Delta(g) \in \OO_K^\times$, which implies (b).

  Conversely, if (b) holds, then there exists a $g \in \OO_K[x]$, equivalent to
  $f$, such that $\Delta(g) \in \OO_K^\times$. Let $X \iso \PP^1_K$ be the
  isomorphism that sends $D$ to the zero divisor of $g$. It extends to an
  isomorphism $\X \iso \PP^1_{\OO_K}$, for a unique smooth $\OO_K$-model $\X$
  of $X$. The assumptions on $g$ now imply that the closure of $D \cup \{
  \infty \} \subset X$ in $\X$ is étale over $\Spec \OO_K$, and so (a) holds.
\end{proof}

Proposition \ref{prop:goodredbq} gives us a second way to optimize a short
Weierstrass equation \eqref{eq:shortweq} of a given Picard curve.

\begin{corollary}\label{cor:binred}
  Let $(K, v)$ be a discretely valued field and assume that both $\Char(K) \neq
  3$ and $\Char(k) \neq 3$. Every nonspecial Picard curve $Y/K$ admits an
  equation
  \begin{equation}\label{eq:binredweq}
    Y :\; y^3 = c f_0 (x) ,
  \end{equation}
  where $0 \le v (c) \le 2$ and $f_0\in \OO_K[x]$ is a reduced quartic.
\end{corollary}

\begin{proof}
  Let $y^3 = f(x)$ with $f(x)\in K[x]$ be a short Weierstrass equation
  for $Y$, which exists by Theorem \ref{thm:shortweq}.  Scaling the
  coordinates if necessary, we may assume that $f = \gamma g$, where
  $\gamma \in \OO_K$ and where $g$ is monic.

  Let $f_0$ be the reduced quartic equivalent to $g$. We have
  \begin{equation}
    g (\alpha x + \beta) = \gamma \alpha^{-4} f_0 .
  \end{equation}
  Now $\delta = \gamma \alpha^{-4}$ is integral since $v (\alpha) \le 0$. Let
  $e = \lfloor v (\delta) / 3 \rfloor$.  Set $c = \delta \pi^{-3 n}$. By Lemma
  \ref{lem:isos} $(x, y) \mapsto (\alpha x + \beta, \pi^n y)$ defines an
  isomorphism between $Y$ and the curve defined by \eqref{eq:binredweq}
  satisfying the requirement of the statement.
\end{proof}

\begin{definition}\label{def:binredweq}
  We call an equation for $Y$ as in Corollary \ref{cor:binred} a \emph{reduced
  short Weierstrass equation} for $Y$.
\end{definition}

\begin{remark}\label{rem:red_min}
  Reduced short Weierstrass equations are not always minimal, as is shown by
  the example
  \begin{equation}
    Y_1 :\; y^3 = 7 (x^4 - 9 x^2 - 10 x - 9)
  \end{equation}
  for $p = 7$. This reduced equation has discriminant valuation $19$. However,
  the short Weierstrass equation
  \begin{equation}
    Y_2 :\; y^3 = 49 x^4 - 56 x^3 + 15 x^2 + 2 x - 1
  \end{equation}
  defines an isomorphic curve and has discriminant valuation $10$.
\end{remark}

\subsection{Criteria for good reduction}\label{sec:goodred2}

We can characterize good reduction of nonspecial Picard curves by their
reduced short Weierstrass model as long as the residue characteristic does not
equal $3$.

\begin{proposition}\label{prop:goodredbinary}
  Let $(K, v)$ be a complete discretely valued field with $\Char (K) \neq 3$
  whose residue field does not have characteristic $3$. Let $Y$ be a nonspecial
  Picard curve over $K$ with reduced short Weierstrass equation
  \begin{equation}\label{eq:goodredbinary}
    Y :\; y^3 = c f .
  \end{equation}
  Then $Y$ has good reduction modulo $v$ if and only if $v (c) = 0$ and $v
  (\Delta (f)) = 0$.
\end{proposition}

\begin{proof}
  For the proof we may replace $K$ by an unramified extension and therefore we
  may assume that $K$ contains a primitive third root of unity $\zeta_3$.
  Assume that the hypotheses on $c$ and $f$ hold. Let $\Y \subset
  \PP^2_{\OO_K}$ be the plane quartic over $\OO_K$ obtained by homogenizing
  \eqref{eq:goodredbinary}. Using Lemma \ref{lem:disc_nonspecial} and
  Proposition \ref{prop:goodredternary}, we see that $\Y$ is smooth over
  $\OO_K$, so that $Y$ has good reduction.

  For the converse, let $y^3 = c f(x)$ be a reduced short Weierstrass equation
  for $Y$. Let $G$ be the distinguished group of automorphisms of $Y$ generated
  by the automorphism $\sigma:y\mapsto \zeta_3 y$. Let $L$ be a minimal Galois
  extension of $K$ over which $Y$ has semistable reduction. Then by
  \cite[Section 4.1]{Picard1}, the curve $Y_L$ has a unique stable
  $\OO_L$-model $\Y$. Moreover, the action of $G$ extends to $\Y$ and the
  quotient scheme $\X = \Y / G$ is a semistable model of $Y_L / G = \PP^1_L$.
  If $\hat{\D} \subset \X$ denotes the closure of the branch divisor $\hat{D} =
  (f)_0 \cup \{ \infty \}$, then $(\X, \hat{\D})$ is the stably marked model of
  $(\PP^1_K, \hat{D})$. Let us denote by $\Yb$ and $\Xb$ the special fibers of
  $\Y$ and $\X$. The induced map $\Yb \to \Xb$ is called the \emph{stable
  reduction} of the cover $Y \to \PP^1_K$. Lemma 4.1 and Theorem 4.2 of
  \cite{Picard1} give a precise description of this map, distinguishing 5 cases
  (a)-(e). Case (a) is the only case where $\Yb$ is smooth. In this case, $\Xb$
  is smooth as well.

  Assume that $Y$ has good reduction over $K$. Then $Y$ has semistable
  reduction over $L = K$, and its stable model $\Y$ is smooth over $\OO_K$.
  Therefore, we are in Case (a) of Lemma 4.1 and Theorem 4.2 from
  \cite{Picard1}. It follows that $\X = \Y/G$ is a smooth model of $\PP^1_K$
  such that the closure $\hat{\D}$ in $\X$ of the branch divisor $\hat{D}$ is
  \'etale over $\Spec\OO_K$. By Proposition \ref{prop:goodredbq} this implies
  that the discriminant of $f$ is a unit in $\OO_K$, \ie, that the reduction
  $\bar{f}\in k[x]$ of $f$ is separable. It also follows from the proof of
  \cite[Theorem 4.1]{Picard1} that $v_K (c) = 0$; a more detailed argument in a
  general setting is given in the proof of \cite[Proposition 4.5]{superell}.
\end{proof}

\subsection{Invariants and twists}\label{sec:invariants}

Let $K$ be a field with $\Char (K) \neq 3$, and let $\Kbar$ be an algebraic
closure of $K$. In general, the isomorphism classes of Picard curves over
$\Kbar$ can be classified by using the Dixmier--Ohno invariants \cite{LLR}. The
definition of these invariants uses classical invariants of ternary quartics
and is relatively complicated. If we suppose additionally that $\Char (K) \neq
2$, as we do throughout this section, then we can describe the isomorphism
classes of nonspecial Picard curves over $\Kbar$ and $K$ by using a smaller and
more elementary set of invariants than the full set of Dixmier--Ohno
invariants.

\begin{proposition}\label{prop:invs}
  Let $Y$ be a nonspecial Picard curve over $K$. Then $Y$ admits an equation
  \begin{equation}\label{eq:invs}
    Y :\; y^3 = x^4 + c_2 x^2 + c_3 x + c_4 .
  \end{equation}
  The isomorphism class of $Y$ over $K$ is determined by the point $(c_2 : c_3
  : c_4)$ in the weighted projective space $\PP (6 : 9 : 12) (K)$.
\end{proposition}

\begin{proof}
  By Theorem \ref{thm:shortweq}, $Y$ admits a short Weierstrass equation
  \eqref{eq:shortweq} over $K$, from which we can obtain \eqref{eq:invs} via a
  Tschirnhausen transformation.

  Lemma \ref{lem:isos} implies that the only possible isomorphisms between two
  curves $Y$ and $Y'$  of the form \eqref{eq:invs} are of the form $(x, y)
  \mapsto (\lambda x, \mu y)$, with $\lambda^4 = \mu^3$. Writing out the
  conditions for this to be an isomorphism, we end up with
  \begin{equation}
    (c'_2, c'_3, c'_4)
    = (\lambda^2 \mu^{-3} c_2, \lambda \mu^{-3} c_3,\lambda \mu^{-3} c_4)
    = (\lambda^{-2} c_2, \lambda^{-3} c_3, \lambda^{-4} c_4) .
  \end{equation}
  Since we need $\lambda^4 = \mu^3$, we see that $\lambda = (\mu / \lambda)^3$
  has to be a third power. If we write $\nu = \mu / \lambda$, then we have
  \begin{equation}
    (c'_2, c'_3, c'_4)
    = (\nu^{-6} c_2, \nu^{-9} c_3, \nu^{-12} c_4)
  \end{equation}
  Conversely, if such a $\nu$ exists, then we can take $\lambda = \nu^3$, $\mu
  = \nu^4$.
\end{proof}

\begin{remark}
  Proposition \ref{prop:invs} implies that we can represent isomorphism classes
  of nonspecial Picard curves over both $\Kbar$ and $K$ by the unique weighted
  representatives of points in weighted projective space that were defined in
  \cite[Section 1]{LR11}. We have used these representatives for the nonspecial
  curves in our database described in Appendix \ref{sec:database}, as they give
  us an effective criterion for the isomorphism of a given curve with one in
  the database.
\end{remark}

\begin{remark}
  Let $(K, v)$ be a discretely valued field whose residue field $k$ has
  characteristic not equal to $2$ or $3$. We can then apply the Tschirnhausen
  transformation to both the curve $Y$ and its reduction modulo $v$.
  Proposition \ref{prop:goodredbinary} then yields the following simple
  criterion for good reduction. Let $\Delta$ be the discriminant of the
  polynomial $x^4 + c_2 x^2 + c_3 x + c_4$ in \eqref{eq:invs}. Then $Y$ has
  good reduction (resp.\ potentially good reduction)  if and only if the point
  $(c_2, c_3, c_4, \Delta)$ admits an integral representative in $\PP (6 : 9 :
  12 : 36) (K)$ (resp.\ in $\PP (6 : 9 : 12 : 36) (\Kbar)$) that reduces to a
  point whose final coordinate is nonzero. A complete criterion that includes
  the cases of residue characteristic $2$ and $3$ is given in \cite{LLR}.
\end{remark}

\begin{remark}
  Proposition \ref{prop:invs} allows us to give a complete description of the
  twists of nonspecial Picard curves, which were already classified in work by
  Lorenzo García \cite{Lorenzo}. They are as follows.

  Again let $K$ be a field of characteristic $\neq 2, 3$, and let $Y:\; y^3 =
  x^4 + c_2 x^2 + c_3 x + c_4$ be a nonspecial Picard curve over $K$.
  Generically, $\Aut_{\Kbar} (Y_{\Kbar}) \simeq \ZZ/3\ZZ$, in which case the
  isomorphism classes of twists of $Y$ correspond bijectively to the elements
  of the quotient $K^\times / (K^\times)^3$ by assigning to $\lambda \in
  K^\times$ the twist
  \begin{equation}
    Y_{\lambda} :\;
    y^3 = x^4 + \lambda^2 c_2 x^2 + \lambda^3 c_3 x + \lambda^4 c_4 .
  \end{equation}

  We have $\Aut_{\Kbar} (Y_{\Kbar}) \simeq \ZZ/6\ZZ$ if and only if $c_3 = 0$
  and $c_2 \neq 0$. In this case the twists of $Y$ correspond bijectively to
  $K^\times / (K^\times)^6$ via
  \begin{equation}
    Y_{\lambda} :\; y^3 = x^4 + \lambda c_2 x^2 + \lambda^2 c_4 .
  \end{equation}

  Finally, we have $\Aut_{\Kbar} (Y_{\Kbar}) \simeq \ZZ/9\ZZ$ if and only if
  $c_2 = c_4 = 0$. In this case the twists of $Y$ correspond bijectively to
  $K^\times / (K^\times)^9$ via
  \begin{equation}
    Y_{\lambda} :\; y^3 = x^4 + \lambda c_3 x .
  \end{equation}

  Given an integral polynomial $f$ defining a nonspecial curve $Y$ over a
  number field $K$, along with a finite set $S$ of primes of $K$, we can
  quickly determine the twists of $Y$ with good reduction outside $S$, since
  for this we need only consider classes represented by $\lambda$ having
  trivial valuation outside the primes in $S$ and the primes dividing the
  discriminant of $f$. Indeed, considering the twists above, an invocation of
  Proposition \ref{prop:goodredbinary} shows that nontrivially twisting a
  Picard curve at a prime where it has good reduction yields curves that no
  longer have good reduction at this prime.

  Twists of the standard special Picard curve and their reduction properties
  will be discussed in Section \ref{sec:classification}. Special Picard curves
  form a single $\Kbar$-isomorphism class, which corresponds to the point $(0 :
  0 : 1)$ in Proposition \ref{prop:invs}.
\end{remark}

\section{Special Picard curves}\label{sec:special}

In this section we turn our attention to special Picard curves (Definition
\ref{def:special}) and extend most of our previous results for nonspecial
curves to them. Moreover, we give a complete list of all special Picard curves
over $\QQ$ with good reduction outside the primes $p=2,3$.

The situation is more complex than for nonspecial Picard curves, due to the
larger automorphism group. Special Picard curves can always be written as
superelliptic curves of exponent $4$, \ie, they admit an equation of the form
\begin{equation}
    x^4 = g(y),
\end{equation}
where $g$ is a quartic polynomial which is $\PGL_2$-equivalent, over the
algebraic closure of $K$, to the cubic polynomial $y^3+1$. Though it is easy to
write down a versal family of such polynomials, writing down the finite list of
all twists with good reduction outside a finite set of primes turns out to be a
surprisingly subtle problem.

Throughout, $K$ denotes a field of characteristic $\neq 2,3$. We choose an
algebraic closure $\Kbar/K$ and set $\Gamma=\Gal(\Kbar/K)$. We also choose a
primitive $12$-th root of unity $\zeta_{12}\in \Kbar$ and set $\zeta_4 :=
\zeta_{12}^3$ and $\zeta_3 := \zeta_{12}^4$.

\subsection{The automorphism group of $Y_0$} \label{sec:special_G}

Recall from Lemma \ref{lem:special} that there is a unique special Picard curve
$Y_0$ over $\Kbar$. We write it as a superelliptic curve of exponent $4$:
\begin{equation} \label{eq:Y_0}
   Y_0 :\; x^4 = y^3 + 1.
\end{equation}
We let $G:=\Aut_{\Kbar}(Y_0)$ denote its automorphism group. It contains the
three elements $\sigma,\tau,\rho\in G$ defined as follows:
\begin{equation} \label{eq:sigma_tau_rho}
  \sigma (x, y) = (x, \zeta_3 y), \qquad
  \tau (x, y) = (\zeta_4 x, y), \qquad
  \rho (x, y) = \left( \frac{\sqrt{3} x}{y + 1},
       \frac{-y + 2}{y + 1}\right).
\end{equation}
Here $\sqrt{3}$ is chosen such that $\zeta_4\sqrt{3}=2\zeta_3+1$.

Let
\begin{equation}
   \psi_0 : Y_0\to\PP^1_{\Kbar}
\end{equation}
denote the morphism defined by the element $y$ of the function field of $Y_0$.
Then $\psi_0$ is a cyclic Galois cover of degree $4$, with Galois group
generated by $\tau$. The branch locus is the divisor
\begin{equation}\label{eq:D0}
    D_0 = (y^3+1)_0\cup\{\infty\}\subset\PP^1_K, \quad
    \text{so}\quad D_0(\Kbar)= \{\infty, -1, -\zeta_3, -\zeta_3^2 \},
\end{equation}
and $\psi_0$ is totally branched over each of these four points. Note that the
four monodromy generators are all equal (\ie, $\psi_0$ has type $(1,1,1,1)$, in
the notation of Remark \ref{rem:distpoint}).

\begin{lemma}\label{lem:special_aut}
  \begin{enumerate}[(a)]
  \item
    The group $G$ has order $48$ and is generated by the three elements
    $\sigma, \tau, \rho$ defined in \eqref{eq:sigma_tau_rho}. (It has label
    $[48; 33]$ in \texttt{GAP}'s database of small groups.)
  \item
    The center $Z(G)$ of $G$ is cyclic of order $4$, generated by $\tau$. The
    quotient $\overline{G} := G / Z(G)$ can be identified, via the map
    $\psi_0$, with the subgroup of $\Aut_{\Kbar}(\PP^1_{\Kbar}) =
    \PGL_2(\Kbar)$ that fixes the branch divisor $D_0$. After labeling the
    points of $D_0$, the induced permutation representation of $\overline{G}$
    on $D_0$ induces an isomorphism $\overline{G}\cong A_4$.
  \item
    There are precisely four subgroups of $G$ of order $3$. They are all
    conjugate to the subgroup $\gen{\sigma}$, and each of them fixes exactly
    one of the $4$ fixed points of $\tau$.
  \end{enumerate}
\end{lemma}

\begin{proof}
  Statements (a) and (b) are consequences of \cite[Theorem 6.1]{param}.
  Statement (c) can be shown by a calculation.
\end{proof}

The group $G$ has a natural linear action on $V:=H^0(Y_0, \Omega_{Y_0}^1)$. To
describe this action, we choose a basis
\begin{equation}
  \omega_1 = \frac{x\,{\rm d}x}{y^2}, \quad
  \omega_2 = \frac{{\rm d}x}{y}, \quad
  \omega_3 = \frac{{\rm d}x}{y^2}
\end{equation}
of $V$ as $\Kbar$-vector space.  This representation decomposes into a direct
of sum of two irreducible subrepresentations, as follows:
\begin{equation}\label{eq:decomp}
  V=V_1\oplus V_2, \quad \text{ where }\quad
  V_1=\langle \omega_1\rangle_{\Kbar} \quad \text{ and }\quad
  V_2=\langle \omega_2, \omega_3\rangle_{\Kbar}.
\end{equation}
Indeed, the center $Z(G)=\gen{\tau}$ acts via a character of order $2$ on $V_1$
and via a character of order $4$ on $V_2.$ Note that $G$ acts faithfully on
$V_2$.

The canonical embedding of $Y_0$ is described by $Y_0\hookrightarrow
\PP(V^\ast),$ where $V^\ast$ is the vector space dual to $V$. The composition
\begin{equation}\label{eq:canonical}
  Y_0\hookrightarrow \PP(V^\ast) \dasharrow \PP(V_2^\ast),
\end{equation}
with the natural projection on $\PP(V_2^\ast)$ can be identified with the map
$\psi_0:Y_0\to \PP^1_{\Kbar}$ corresponding to $y$. Indeed, we have
$y=\omega_2/\omega_3$. In the following, we will often identify $G$ with its
image under the embedding
\begin{equation}\label{eq:Gmatrix}
  G\hookrightarrow \GL(V_2^\ast)\simeq \GL_2(\Kbar)
\end{equation}
obtained by the representation $V_2^\ast$. Then we obtain an embedding of two
short exact sequences:
\begin{equation} \label{eq:two_ses}
  \xymatrix{
     1 \ar[r] & Z(G) \ar[r] \ar@{^{(}->}[d] & G \ar[r] \ar@{^{(}->}[d]
                         & \overline{G} \ar[r] \ar@{^{(}->}[d] &  1 \\
     1 \ar[r] & \Kbar^\times \ar[r] & \GL_2(\Kbar) \ar[r]  &
                                 \PGL_2(\Kbar) \ar[r] &  1
  }
\end{equation}
Here the embedding $\iota : \overline{G}\hookrightarrow \PGL_2(\Kbar)$ comes
from the action on the quotient curve $Y_0/Z(G)\cong\PP^1_{\Kbar}$, see Lemma
\ref{lem:special_aut}.(b).

\begin{remark}\label{rem:A4tilde}
  For $g\in G$ one may compute the characteristic polynomial of its image
  $\iota (g)\in \GL_2(\Kbar)$, either by a direct computation, or by using the
  character table of $G$. One finds that $\det(\iota(g))\in \{\pm 1\}$. One
  checks that the index-$2$ subgroup $G_1:=G\cap \SL_2(\Kbar)$ is isomorphic to
  $\tilde{A}_4$. Here $\tilde{A}_4$ is the unique nontrivial central extension
  of $A_4$ by $\{\pm 1\}$.

  The image $\overline{\rho}$ of the element $\rho$ in $\overline{G}\simeq A_4$
  corresponds to a permutation with cycle type $(2, 2)$. It interchanges $-1$
  with $\infty$ and $-\zeta_3$ with $-\zeta_3^2$. The $4$ lifts $\tau^i\rho$ of
  $\overline{\rho}$ to $G$ have order $2$ if $i\equiv 0\pmod{2}$ and $4$ if
  $i\equiv 1\pmod{2}$. The lifts of order $4$ are in $G_1$, and the lifts of
  order $2$ are not.

%  A more geometric way of seeing that $M(\rho)$ has characteristic
%  polynomial $t^2-1$ is to note that $E_1:=Y_0/\langle \rho\rangle$ is
%  an elliptic curve. This implies that $\rho$ has eigenvalues $1,
%  -1$ on $V_1$. Indeed, the vector space  $V_1^{\langle\rho\rangle}=H^0(E_1,
%  \Omega)$ has dimension $1$. Since $\tau$ induces an automorphism of
%  order $4$, we have $j(E_1)=1728$.
%
%  Similarly, we may interprete the subspace $V_2=H^(E_0, \Omega),$ where
%  $E_0=Y_0/\langle \tau^2\rangle$ is an elliptic curve wit $j(E_0)=0.$
%  One may show that the Jacobian of $Y_0$ decomposes up to isogeny as a
%  product of elliptic curves:
%  \begin{equation}
%  J(Y_0)\sim  E_0\times E_1^2.
%  \end{equation}
\end{remark}

\subsection{Descent for special Picard curves} \label{sec:special_descent}

We let $Y_{0,K}$ denote the $K$-model of $Y_0$ given by the equation
\eqref{eq:Y_0}. This is called the \emph{standard model} of $Y_0$. There is a
natural identification $Y_0=Y_{0,K}\otimes_K\Kbar$, which induces a semilinear
action of $\Gamma$ on $Y_0$. This action may be regarded as a section
$s_0:\Gamma\to \Aut_K(Y_0)$ of the short exact sequence
\begin{equation}\label{eq:descent1}
  1 \to G \to \Aut_K(Y_{0}) \to \Gamma \to 1.
\end{equation}
The section $s_0$ defines a continuous action of $\Gamma$ on $G$, by
conjugation. We call it the \emph{standard action}, and we will henceforth
regard the group $G$ as a \emph{$\Gamma$-group} in the sense of \cite[\S
5.1]{SerreCG}, using the standard action.

Recall from Lemma \ref{lem:special} that any special Picard curve $Y$ over $K$
admits an isomorphism
\begin{equation}
   Y \otimes_K \Kbar \cong Y_0.
\end{equation}
The $\Gamma$-action on $Y_0$ induced by such an isomorphism corresponds to
another section $s:\Gamma\to\Aut_K(Y_0)$ of \eqref{eq:descent1}.  The
`difference' between $s$ and $s_0$ defines a $1$-cocycle
\begin{equation}\label{eq:cocycle}
  (A_\gamma)_{\gamma\in \Gamma}\in Z^1(\Gamma, G),\qquad \text{ where }
   A_\gamma=s(\gamma)\circ s_0(\gamma)^{-1}.
\end{equation}
Its class in $H^1(\Gamma, G)$ only depends on the $K$-model $Y$ of $Y_0$. This
correspondence yields a bijection between the set of isomorphism classes of
special Picard curves over $K$ and the Galois cohomology set $H^1(\Gamma, G)$,
see, \eg,~\cite[Chapter V, \S 4]{SerreAGCF}.

As a first application of descent theory we  prove that every special Picard
curve can be written as a superelliptic curve of exponent $4$, with some kind
of normal form. A slightly different approach to this question can be found in
\cite[Section 5]{Lorenzo}.

\begin{theorem}\label{thm:special_sweq}
  Let $Y$ be a special Picard curve over a field $K$ of characteristic $\neq
  2,3$.
  \begin{enumerate}[(a)]
    \item The curve $Y$ admits a defining equation
      \begin{equation}\label{eq:special2}
        Y :\; x^4 = a g (y), \qquad  g(y)=y^4 + b y^2 + c y + d\in K[y]
      \end{equation}
      with $a\in K^\times$ and  $\Delta (g)\neq 0$.
    \item An equation \eqref{eq:special2} as in (a) defines a special Picard
      curve if and only if
      \begin{equation}
        b^2+12d=0.
      \end{equation}
  \end{enumerate}
\end{theorem}

\begin{proof}
  Let $Y$ be a special Picard curve over $K$. Let $s:\Gamma\to \Aut_K(Y_0)$ be
  the section corresponding to $Y$ and the choice of an isomorphism
  $Y\otimes_K\Kbar \cong Y_0$. Let $(A_\gamma)_\gamma\in Z^1(\Gamma, G)$ be the
  corresponding cocycle as in \eqref{eq:cocycle}. By abuse of notation, we also
  write $(A_\gamma)_{\gamma}$ for the corresponding cocycle in
  $Z^1(\Gamma,\GL_2(\Kbar))$ obtained via the embedding $G \hookrightarrow
  \GL_2(\Kbar)$. The latter describes the twisted $\Gamma$-action on the
  $G$-subrepresentation $V_1(s)\subset V(s) := H^0(Y_{\Kbar},
  \Omega_{Y_{\Kbar}})$, if we identify the underlying $\Kbar$-vector space with
  the subrepresentation $V_1\subset V=H^0(Y_0,\Omega_{Y_0})$, via the chosen
  isomorphism $Y\otimes_K\Kbar \cong Y_0$ (see \eqref{eq:decomp}).  By Hilbert
  90 we have that $H^1(G, \GL_2(\Kbar))=1$, see, \eg,~\cite[Prop.~3 in \S
  X.1]{SerreCL}. It follows that
  \begin{equation} \label{eq:resolve_cocycle}
    A_\gamma=C^{-1}\gamma(C) \qquad \text{ for some }C\in \GL_2(\Kbar).
  \end{equation}
  Then \[ (\omega_2,\omega_3)\cdot C = (\eta_2,\eta_3) \] defines a basis
  $(\eta_2,\eta_3)$ of $V_2(s)$, which is $\Gamma$-invariant. This means that
  we may regard $\eta_2,\eta_3$ as differentials on $Y$. The quotient
  $y:=\eta_2/\eta_3$ is a rational function on $Y$ and hence induces a map
  \begin{equation}
    \psi:Y\to\PP^1_K.
  \end{equation}
  By construction, this is a $K$-model of the map $\psi_0:Y_0\to \PP^1_{\Kbar}$
  defined in \S \ref{sec:special_G}.

  By the same argument, we can also choose a $K$-basis $\eta_1$ of the
  $1$-dimensional subrepresentation $V_1(s)\subset V(s)$. Then
  $x:=\eta_1/\eta_3$ is a rational function on $Y$ such that $\tau^*x=\zeta_4
  x$, where we regard $\tau$ as an automorphism of $Y\times_K\Kbar$. Because
  $y$ is fixed under the action of $\tau$, it follows that the curve $Y$ admits
  an affine equation
  \begin{equation}
    x^4 = f(y),
  \end{equation}
  where $f$ is a rational function in $y$. Since the quotient map given by $(x,
  y) \to y$ has identical monodromy generator everywhere by the remark after
  \eqref{eq:D0}, we may assume, replacing $x$ by $1/x$ if necessary, that $f$
  is given by a separable polynomial of degree $4$.

  Choosing different $K$-bases $(\eta_2,\eta_3)$ and $\eta_1$ corresponds to a
  change of coordinates
  \begin{equation}
    y\mapsto \frac{\alpha y+\beta}{\gamma y+\delta}, \quad
    x\mapsto \epsilon x.
  \end{equation}
  We may assume, after a suitable change of coordinates as above, that $Y$ is
  given by a Weierstrass equation of the form
  \begin{equation}
      x^4 = ag(y), \qquad g = y^4 + by^2 + cy + d,
  \end{equation}
  with $a,b,c,d\in K$, $a\neq 0$ and $\Delta(g)\neq 0$. This proves (a).

  For later use, we remark that the branch divisor $D=(g)_0\subset\PP^1_K$ of
  the map $\psi$ is equal, by construction, to
  \begin{equation}
    D = C(D_0),
  \end{equation}
  where $D_0=\{\infty,-1,-\zeta_3,-\zeta_3^2\}$ is the branch divisor of
  $\psi_0$ and $C$ is the matrix from \eqref{eq:resolve_cocycle}, considered as
  an element of $\PGL_2(\Kbar)=\Aut_{\Kbar}(\PP^1_{\Kbar})$.

  \bigskip To prove (b) we note that a superelliptic curve $Y$ given by an
  equation $x^4=ag(y)$ as in (a) is a twist of $Y_0$ if and only if there
  exists $C\in \PGL_2(\Kbar)$ such that $D=C(D_0)$, where $D$ is the set of
  roots of $g$. It is well known that this is the case if and only if the
  $j$-invariants of $D$ and $D_0$ are equal. Here the $j$-invariant of a
  divisor $D$ is defined as
  \begin{equation}
    j(D) := 256\frac{(\lambda^2-\lambda + 1)^3}{\lambda^2(\lambda-1)^2},
  \end{equation}
  where $\lambda\in\Kbar$ is the cross ratio of the four
  $\Kbar$-points of $D$. Note that $j(D_0)=0$. A straightforward
  computation shows that $j(D)=j(D_0)$ if and only if $ b^2 + 12 d =
  0$. This concludes the proof of the theorem.
\end{proof}

\begin{definition}
  A \emph{special polynomial} is a polynomial of the form
  \begin{equation}\label{eq:specialpoly}
    g(y) = y^4 + 6 b y^2 + c y - 3 b^2 \in K[y]
    \text{ with } \Delta(g) \neq 0.
  \end{equation}
\end{definition}

\begin{remark}\label{rem:special_isos}
  The proof of Theorem \ref{thm:special_sweq} also yields a
  description of the isomorphisms between special Picard curves,
  analogous to the description in Lemma \ref{lem:isos} in the
  nonspecial case. In fact, we see that two special Picard curves
  $Y:x^4=ag(y)$ and $Y':x^4=a'g'(y)$ are isomorphic over $K$ if and
  only if there exists
  \begin{equation}
    A = \begin{pmatrix} \alpha&\beta\\ \gamma&\delta\end{pmatrix} \in\GL_2(K)
  \end{equation}
  and $\epsilon\in K^\times$ such that
  \begin{equation}
    a'g' = \epsilon^4 a g\left(\frac{\alpha x+\beta}{\gamma x+\delta}\right)
    (\gamma x+\delta)^4.
  \end{equation}
  The image of a special polynomial under an element of $\PGL_2(K)$ need not be
  special. However, each equivalence class is represented by a special
  polynomial. Identifying a special polynomial with the divisor corresponding
  to its roots allows us to interpret $g\circ A$ also if $A(\infty)$ is a root
  of $g$.
\end{remark}

\begin{lemma}\label{lem:specialp=3}
  Let $K$ be a field of characteristic $\neq 2,3$.  Let $g(y) = y^4 + 6 b y^2 +
  c y - 3 b^2 \in K[y]$ be a special polynomial and $L/K$ be a Galois extension
  that contains the splitting field of $g$.  Then $\gamma\in \Gal(L/K)$ acts as
  an odd permutation on the roots of $g$ if and only if
  $\gamma(\zeta_3)=\zeta_3^2$.
\end{lemma}

\begin{proof}
  The statement of the lemma holds for the branch divisor $D_0$ of $Y_{0, K}$.

  Let $g(y)\in K[y]$ be an arbitrary special polynomial and $D$ its divisor of
  roots. Any element $\gamma\in \Gal(L/K)$ acts both on $D$ and on $D_0$. The
  permutation representation of $\rho$ acting on $D$ differs from that on $D_0$
  by $C^{-1}\gamma(C)$ for some $C\in \PGL_2(L)$ by the proof of Theorem
  \ref{thm:special_sweq}.(a). The statement of the lemma follows, since
  $C^{-1}\gamma(C)\in\overline{G}\simeq A_4$, by Lemma
  \ref{lem:special_aut}.(b).
\end{proof}

\subsection{Reduction of special Picard curves}\label{sec:red_special}

In this section we characterize the special Picard curves with good reduction
to characteristic $p$. The stable reduction of the special Picard curve
$Y_0/\QQ$ \eqref{eq:special} has been computed in \cite[Section
5.1.3]{MichelDiss}. (See also \cite[Example 5.6]{Picard1}.) This result implies
that every special Picard curve $Y$ has bad reduction to characteristic $2$
over any field extension of $K$ and has potentially good reduction to
characteristic $p\neq 2$.

We assume that $(K,v)$ is a discretely valued field of mixed characteristic
zero. For the results of this section it is no restriction to assume that its
residue field $k$ is algebraically closed.  We write $\OO_K$ for the valuation
ring and $\pi$ for the uniformizer.  The following proposition treats the cases
of residue characteristic $p=2,3$. The result for $p=3$ extends Proposition
\ref{prop:badredwild} to special Picard curves. The proof given here follows
that from \cite[Prop.~3.4]{Picard1}.

\begin{proposition}\label{prop:specialp=3}
  Let $Y$ be a special Picard curve over a discretely valued field $(K, v)$ of
  mixed characteristic zero.
  \begin{enumerate}[(a)]
    \item If $3$ has odd valuation for $v$, then $Y$ has bad reduction at $v$.
    \item If the residue characteristic is $p=2$, then $Y$ has bad reduction
      over any extension of $K$.
  \end{enumerate}
\end{proposition}

\begin{proof}
  We choose an equation \eqref{eq:special2} for $Y/K$.  Let $L/K$ be a Galois
  extension such that $Y$ has good reduction over $L$. After possibly extending
  $L$, we may assume that $L$ contains the splitting field of $g$ and a
  primitive $3$rd root of unity $\zeta_3$.  Let $\mathcal{Y}$ denote the stable
  model of $Y_{\Kbar}$. Define $\mathcal{W} = \mathcal{Y} / \langle \tau
  \rangle$ and write $\overline{Y}$ and $\overline{W}$ for the special fibers
  of $\mathcal{Y}$ and $\mathcal{W}$, respectively.

  The branch divisor $D$ of $\psi:Y\to Y/\langle \tau\rangle=:W\cong \PP^1_y$
  extends to an \'etale divisor over $\mathcal{W}$. The image $\overline{D}$ of
  $D$ in $\overline{W}$ therefore consists of $4$ distinct $k$-rational points.
  Let $\gamma\in \Gal(L/K)$ be such that $\gamma(\zeta_3)=\zeta_3^2$. Such an
  element exists, by our assumption on $K$.  Lemma \ref{lem:specialp=3} implies
  that $\gamma$ acts nontrivially on $\overline{D}$ and hence on
  $\overline{Y}$.  As in the proof of \cite[Prop.~3.4]{Picard1}, we conclude
  that there does not exist a smooth model of $Y$ over $K$.

  Statement (b) follows from \cite[Section 5.1.3]{MichelDiss}.
\end{proof}

\begin{proposition}\label{prop:special_reduction}
  Let $Y/K$ be a special Picard curve given by an equation $x^4 = a g(y)$.
  Assume that the residue characteristic is different from $2$. Then $Y$ has
  good reduction at $v$ if and only if
  \begin{enumerate}[(a)]
  \item
    $v(a)\equiv 0 \pmod{4}$, and
  \item
    the splitting field of $g$ is unramified at $v$.
  \end{enumerate}
\end{proposition}

Note that Proposition \ref{prop:specialp=3}.(a) implies that the conditions in
Proposition \ref{prop:special_reduction} are never satisfied in the case that
$K/\QQ_3^{\nr}$ is unramified. We refer to Example \ref{exa:specialp=3} for a
closer consideration of such a case of bad reduction.

\begin{proof}
  Assume  that the conditions (a) and (b) are satisfied. Since we assume that
  $K$ is complete with respect to $v$ and  that the residue field $k$ is
  algebraically closed \cite[Cor.~4.6]{superell} implies that there exists a
  stable model of $Y$ over $\OO_K$. Since $Y$ has potentially good reduction,
  it has good reduction over $K$.

  Assume that Condition (a) or (b) is not satisfied, and let $L/K$ be a Galois
  extension that contains both a $4/\gcd(v(a), 4)$-th root of $a$ and the
  splitting field of $g$. Then $Y_{\Kbar}$ has good reduction over $L$. It
  follows from \cite[Section 5]{superell} that the Galois group $\Gal(L/K)$
  acts nontrivially on the reduction $\overline{Y}$ of $Y_{\Kbar}$. (This is
  similar to the argument in the proof of Proposition \ref{prop:specialp=3}.)
  We conclude that $Y$ does not have good reduction over $K$.
\end{proof}

\subsection{Good reduction outside $p=2,3$}
\label{sec:classification}

We will determine all special Picard curves over $\QQ$ with good reduction
outside $p=2,3$ in Theorem \ref{thm:special_2_3}. It follows from Proposition
\ref{prop:specialp=3} that this is the smallest possible set of primes of bad
reduction of a special Picard curves over $\QQ$. The method we use is more
general and can be applied over any number field using any finite set of
primes. The crucial finiteness result one needs is that there are only finitely
many number fields of given degree that are unramified outside a given set of
primes. This result follows directly from Hermite's Theorem (see \cite[Chapter
III, Theorem (2.16)]{NeukirchAZT}). In our situation, the relevant finite list
of number fields can be obtained from the database of number fields by Jones
and Roberts \cite{DatabaseNF}.

By Theorem \ref{thm:special_sweq} any special Picard curve over $\QQ$ is given
by an equation
\begin{equation}
   Y :\; x^4 = ag(y), \quad g = y^4+6by^2+cy-3b^2,
\end{equation}
where $a,b,c\in\ZZ$, $a\neq 0$, and $0\neq \Delta(g) = -3^3(64b^3+c^2)^2$. We
may further assume that $a$ has no fourth power as a nontrivial divisor. Then
Proposition \ref{prop:special_reduction} states that $Y$ has good reduction
outside $p=2,3$ if and only if
\begin{enumerate}[(i)]
  \item
    the splitting field of $g$ is unramified outside $2,3$, and
  \item
    $a = \pm 2^\mu 3^\nu$, with $0\leq \mu,\nu\leq 3$.
\end{enumerate}
We will first find all special polynomials $g$ as in (i), up to isomorphism
(Proposition \ref{prop:pol_list}). Most of the proof is formulated in a more
general setup, which is introduced below. The main result (Theorem
\ref{thm:special_2_3}) is then a direct consequence of Proposition
\ref{prop:pol_list}.

\bigskip
We return to the general assumption of this section. So $K$ is a field of
characteristic $\neq 2,3$ with algebraic closure $\Kbar$, and
$\Gamma:=\Gal(\Kbar/K)$. In addition, we fix a subfield $L\subset\Kbar$ which
is a Galois extension of $K$ containing the third root of unity $\zeta_3$. (In
our final application, $K=\QQ$ and $L/\QQ$ is the maximal extension unramified
outside $2,3$.)

We recall the setup from Sections \ref{sec:special_G} and
\ref{sec:special_descent}.  We denote by $Y_0$ the special Picard curve over
$\Kbar$ defined by \eqref{eq:Y_0} and $Y_{0,K}$ for the $K$-model of $Y$
defined by the same equation. The choice of this model is determined by an
action of $\Gamma$ on $Y_0$, which we call the standard action. The branch
divisor of $\psi_0:Y_0\to Y_0/\langle \tau\rangle\simeq\PP^1$ is denoted by
$D_0$. Note that $D_0$ splits over $L$. We order the points of $D_0(\Kbar)$ as
follows:
\begin{equation} \label{eq:D_0_numbering}
   \alpha_1 := \infty,\; \alpha_2:=-1,\; \alpha_3:=-\zeta_3,\; \alpha_4:=-\zeta_3^2.
\end{equation}
This identifies $\overline{G}=\Aut_{\Kbar}(D_0)\subset\PGL_2(\Kbar)$ with
$A_4$. The exact sequence \eqref{eq:descent1} induces a diagram with exact
rows, where the two lower left vertical arrows describe the actions on the set
of roots:
\begin{equation} \label{eq:descent_ses}
  \xymatrix{
    1 \ar[r] & G \ar[r]\ar@{->>}[d] & \Aut_K(Y_0) \ar[r]\ar@{->>}[d] & \Gamma \ar[r]\ar@{=}[d] & 1 \\
    1 \ar[r] & \overline{G}\ar[r]\ar[d]^{\cong} & \Aut_K(D_0) \ar[r] \ar[d]^{\phi} & \Gamma\ar[r]\ar[d]^{\chi} & 1 \\
    1 \ar[r] & A_4\ar[r]&S_4 \ar[r]^{\sgn} & \{\pm 1\} \ar[r] & 1
  }
\end{equation}
Lemma \ref{lem:specialp=3} implies that $\chi:\Gamma\to\{\pm 1\}$ is the
character given by
\begin{equation}\label{eq:chi}
     \gamma(\zeta_3) = \zeta_3^{\chi(\gamma)}, \qquad \text{ for
     }\gamma\in \Gamma.
\end{equation}
We write $s_0:\Gamma\to \Aut_K(Y_0)$ for the section of $\Aut_K(Y_0)\to \Gamma$
corresponding to the standard action of $\Gamma$ on $Y_0$. It induces a section
$\overline{s}_0:\Gamma\to\Aut_K(D_0)$.

We are interested in describing the set
\begin{equation}\label{eq:defS}
  \mathcal{S} := \{ D\subset\PP^1_K \mid \text{there exists } C\in \PGL_2(\Kbar)\text{ with } \;
        D = C(D_0) \;\text{such that $D$ splits over $L$} \}/\sim,
\end{equation}
where $\sim$ means \emph{modulo the $\PGL_2(K)$-action}. By Theorem
\ref{thm:special_sweq}, every element of $\mathcal{S}$ can be represented by a
divisor $D=(g)_0$, where $g\in K[y]$ is a special polynomial.

Let $D =(g)_0\in \mathcal{S}$ be given. Let $Y$ be the special Picard curve
over $K$ given by $Y :\; x^4=g(y)$; it is a $K$-twist of $Y_{0,K}$ and hence
corresponds to a section $s:\Gamma\to \Aut_K(Y_0)$, up to conjugation by $G$.
Let $\overline{s}:\Gamma\to\Aut_K(D_0)$ be the induced section. Composition of
$\overline{s}$ with the map $\Aut_{K} (D_0) \to S_4$ from Diagram
\eqref{eq:descent_ses} yields a homomorphism $\rho:\Gamma\to S_4$ which
satisfies the following two conditions:
\begin{equation}\label{eq:descent_rho_condition}
      {\sgn}\circ \rho = \chi, \quad \text{and\;\; $\rho$ factors over $\Gal(L/K)$}.
\end{equation}

We call $\rho$ the \emph{homomorphism induced by} $D$. It has the following
concrete interpretation, which does not involve the curve $Y$. By definition
there exists an element $C\in\Aut_{\Kbar}(\PP^1_{\Kbar})=\PGL_2(\Kbar)$ such
that $D=C(D_0)$. In particular, we obtain a bijection between the geometric
points of $D_0$ and those of $D$ and therefore, via \eqref{eq:D_0_numbering}, a
numbering of the four geometric points of $D$. The homomorphism $\rho$
corresponds to the action of $\Gamma$ on these points, with respect to this
numbering. We note that the choice of $C$ is unique up to an element of
$\overline{G}\cong A_4$. It follows that $\rho$ is uniquely determined by
$D\in\mathcal{S}$, up to conjugation by an even permutation.

\bigskip We reverse this construction. Let $\rho:\Gamma\to S_4$ be a
homomorphism satisfying \eqref{eq:descent_rho_condition}. Chasing diagram
\eqref{eq:descent_ses} shows that given $\gamma \in \Gamma$, there is a unique
lift $\alpha$ of $\gamma$ to $\Aut_K (D_0)$ with $\phi (\alpha) = \rho
(\gamma)$. This implies that $\rho$ comes from a unique section $\overline{s}
:\Gamma\to\Aut_K(D_0)$. We thus obtain a cocycle
\begin{equation}
    (A_\gamma)_{\gamma\in\Gamma} \in H^1(\Gamma, \overline{G}), \quad
      A_\gamma:= s(\gamma)\circ s_0(\gamma)^{-1}.
\end{equation}

Let $c(\rho)\in H^1(\Gamma,\PGL_2(\Kbar))$ denote the image of this cocycle
under the natural map
\begin{equation}
    H^1(\Gamma, \overline{G}) \to H^1(\Gamma, \PGL_2(\Kbar)).
\end{equation}
We call $c(\rho)$ the \emph{obstruction class} of $\rho$. Recall also that the
short exact sequence
\begin{equation}
    1 \to \Kbar^\times \to \GL_2(\Kbar) \to \PGL_2(\Kbar) \to 1
\end{equation}
gives rise to an injection
\begin{equation}
    1 = H^1(\Gamma,\GL_2(\Kbar)) \to H^1(\Gamma, \PGL_2(\Kbar)) \to
       H^2(\Gamma, \Kbar^\times).
\end{equation}
The equality follows from Hilbert 90 \cite[Prop.~3, \S X.1]{SerreCL}. We
therefore may consider the obstruction class $c(\rho)$ as element of the Brauer
group $\Br(K)=H^2(\Gamma, \Kbar^\times)$ of the field $K$. A calculation as in
Remark \ref{rem:A4tilde} shows that $c(\rho)$ even lies in
$\Br(K)[2]=H^2(\Gamma, \{\pm 1\})$.

The following proposition states that a homomorphism $\rho$ is induced from an
element $D\in \mathcal{S}$ if and only if the obstruction class $c(\rho)$
vanishes.  In what follows we identify $\rho$ with its equivalence class under
conjugacy with $A_4$. The theory of Weil descent then implies the following.

\begin{proposition} \label{prop:descent_rho}
  There is a bijection between the set $\mathcal{S}$ and the set of
  homomorphisms $\rho:\Gamma \to S_4$ satisfying
  \eqref{eq:descent_rho_condition} and such that $c(\rho)$ is trivial, modulo
  the action of $A_4$ by conjugation. This bijection depends only on the choice
  of the element $D_0\in \mathcal{S}$.
\end{proposition}

Our next goal is to make the obstruction class $c(\rho)$ more explicit. For
this we need some preparation. Recall from \cite[\S 1.5]{SerreWitt} that there
exists a unique central extension
\begin{equation} \label{eq:S_4_extension}
    1 \to \{\pm 1\} \to \tilde{S}_4\to S_4 \to 1
\end{equation}
of $S_4$ with the property that transpositions lift to elements of order $2$ in
$\tilde{S}_4$ and $(2,2)$-cycles lift to elements of order $4$. The extension
\eqref{eq:S_4_extension} corresponds to a certain class
\begin{equation}
     s_4 \in H^2(S_4,\{\pm 1\}).
\end{equation}

For a homomorphism $\rho:\Gamma\to S_4$ we denote by $\rho^\ast s_4$ the
element of $H^2(\Gamma,\{\pm 1\})$ obtained via restriction along $\rho$. The
short exact sequence
\begin{equation}
    1 \to \{\pm 1\} \to \Kbar^\times \to \Kbar^\times \to 1
\end{equation}
induces an injection
\begin{equation}
     1 = H^1(\Gamma, \Kbar^\times) \to H^2(\Gamma,\{\pm 1\}) \to H^2(\Gamma,\Kbar^\times).
\end{equation}
So for a given homomorphism $\rho:\Gamma\to S_4$, we may consider both
$\rho^\ast s_4$ and the obstruction class $c(\rho)$ as elements of $\Br(K)$.

\begin{lemma} \label{lem:descent_obstruction}
  Let $\rho:\Gamma\to S_4$ be a homomorphism satisfying
  \eqref{eq:descent_rho_condition}, and let $c(\rho)\in\Br(K)$ be the
  corresponding obstruction class. Then
  \begin{equation}
       c(\rho) = \rho^\ast s_4.
  \end{equation}
\end{lemma}

\begin{proof}
  Let $\tilde{A}_4\subset\tilde{S}_4$ denote the inverse image of $A_4\subset
  S_4$. These groups fit into the following commutative diagram with exact rows
  and columns:
  \begin{equation} \label{eq:monster_diagram1}
  \xymatrix{
         &       &  1 \ar[d] &  1 \ar[d] &      &             \\
    1 \ar[r] & \{\pm 1\} \ar[r]\ar[d]^= & \tilde{A}_4 \ar[r]\ar[d]
                  & A_4 \ar[d]\ar[r]  &  1                    \\
    1 \ar[r] & \{\pm 1\} \ar[r]  & \tilde{S}_4 \ar[r]\ar[d]
                & S_4 \ar[d]^{\rm sgn}\ar[r]  &  1             \\
         &       & \{\pm 1\}\ar[d] \ar[r]^= & \{\pm 1\}\ar[d] & \\
         &       &          1               &        1        &
  }
  \end{equation}
  Considering the groups in \eqref{eq:monster_diagram1} as $\Gamma$-groups with
  respect to the trivial action of $\Gamma$, we obtain a commutative diagram of
  nonabelian cohomology sets:
  \begin{equation} \label{eq:monster_diagram2}
  \xymatrix{
     H^1(\Gamma,\tilde{A}_4) \ar[r]\ar[d] & H^1(\Gamma,A_4) \ar[r]\ar[d]
                  & H^2(\Gamma,\{\pm 1\}) \ar[d]^=          \\
     H^1(\Gamma,\tilde{S}_4) \ar[r]\ar[d]^{\tilde{\pi}} & H^1(\Gamma,S_4) \ar[r]^{\Delta}\ar[d]^{\pi}
                  & H^2(\Gamma,\{\pm 1\})                   \\
     H^1(\Gamma,\{\pm 1\}) \ar[r]^= & H^1(\Gamma,\{\pm 1\})  &
  }
  \end{equation}

  The rows and columns of \eqref{eq:monster_diagram2} are exact as sequences of
  pointed sets. We remark that an element of $H^1(\Gamma,H)$ for a group $H$
  with trivial $\Gamma$-action is simply a homomorphism $\Gamma\to H$.  In
  particular, a homomorphism $\rho:\Gamma\to S_4$ satisfying
  \eqref{eq:descent_rho_condition} is simply an element of the fiber
  $\pi^{-1}(\chi)\subset H^1(\Gamma,S_4)$. The definitions of the maps in
  \cite[Chapter I, \S 5]{SerreCG} imply that
  \begin{equation} \label{eq:Delta_rho}
     \Delta(\rho) = \rho^*s_4.
  \end{equation}
  We have $\Delta(\rho_0)=\rho_0^*s_4=1$, where $\rho_0 : \Gamma \to S_4$
  corresponds to the standard divisor $D_0\in\mathcal{S}$ defined in
  \eqref{eq:D0}. Indeed, the image of $\rho_0$ is  trivial if $\zeta_3\in K$
  and generated by a transposition otherwise, and the extension $\tilde{S}_4
  \to S_4$ has the property that transpositions lift to elements of order $2$
  in $\tilde{S}_4$.

  Let $(\tilde{A}_4)_{\rho_0}$ and $(A_4)_{\rho_0}$ be \emph{twisted
  $\Gamma$-groups}, as defined in \cite[Chapter I, \S 5.3]{SerreCG}. The theory
  in \loccit\ shows that the fiber $\pi^{-1}(\chi)\subset H^1(\Gamma,S_4)$ to
  which $\rho$ and $\rho_0$ belong is in canonical bijection with
  $H^1(\Gamma,(A_4)_{\rho_0})$. Moreover, it follows from Diagram
  \eqref{eq:monster_diagram1} that we have isomorphisms of $\Gamma$-groups
  \begin{equation}
   G_1 \cong (\tilde{A}_4)_{\rho_0}, \quad \overline{G} \cong (A_4)_{\rho_0} .
  \end{equation}
  compatible with the natural maps $G_1\to\overline{G}$ and $\tilde{A}_4\to
  A_4$. From this we obtain canonical bijections
  \begin{equation}
    \pi^{-1}(\chi) =H^1(\Gamma,(A_4)_{\rho_0})=
    H^1(\Gamma,\overline{G}), \quad \tilde{\pi}^{-1}(\chi) =
    H^1(\Gamma,(\tilde{A}_4)_{\rho_0})=H^1(\Gamma,G_1).
  \end{equation}

  Since $\Delta(\rho_0)=0$,  \cite[Proposition 44]{SerreCG} implies that the
  restriction of the map $\Delta$ to $\pi^{-1}(\chi) =
  H^1(\Gamma,\overline{G})$ equals the map
  \begin{equation}
       H^1(\Gamma,\overline{G}) \to H^2(\Gamma,\{\pm 1\})
  \end{equation}
  induced from the sequence $1\to G_1\to \overline{G}\to 1$.

  Equation \eqref{eq:two_ses} together with Remark \ref{rem:A4tilde} implies
  that the map $G_1\twoheadrightarrow \overline{G}$ is compatible with
  $\GL_2(\Kbar)\twoheadrightarrow \PGL_2(\Kbar)$. It follows that the image of
  $\rho$ under the map
  \begin{equation}
     \pi^{-1}(\chi)= H^1(\Gamma,\overline{G}) \stackrel{\Delta}{\to}
       H^2(\Gamma,\{\pm 1\})
  \end{equation}
  is equal to the obstruction class $c(\rho)$.  Together with
  \eqref{eq:Delta_rho}, this proves the lemma.
\end{proof}

\begin{example} \label{exa:biquadratic}
  Let $K$ be a field not containing $\zeta_3$. We discuss the case that the
  image of $\rho:\Gamma\to S_4$ is an elementary abelian group with $4$
  elements that is not contained in $A_4$ (the \emph{biquadratic case}).  Up to
  conjugation by an element of $A_4$ we may assume that the image is
  \begin{equation}
    \rho(\Gamma) = \gen{(1\,2), (3\, 4)}.
  \end{equation}
  By Galois theory, we obtain a pair $(M_1, M_2)$ of quadratic subextensions of
  $L/K$, namely $M_1=M^{\langle (3\,4)\rangle}$ and $M_2=M^{\langle
  (1\,2)\rangle}$. By Kummer theory we can write $M_i=K[\sqrt{d_i}]$ for some
  $d_i\in K^\times\backslash(K^\times)^2$. Condition
  \eqref{eq:descent_rho_condition} implies that the subextension of $L$
  corresponding to $(1\,2)(3\,4)$ is $K (\zeta_3)$. This means that
  \begin{equation} \label{eq:descent_biquadratic1}
    d_1d_2 \equiv -3 \pmod{(K^\times)^2}.
  \end{equation}

  We now calculate the descent obstruction $c (\rho)$. While this can be done
  explicitly, we instead use Lemma \ref{lem:descent_obstruction} and calculate
  $\rho^\ast s_4$ instead. It follows from the proof of \cite[Lemma 2,
  p.~661]{SerreWitt} that
  \begin{equation} \label{eq:descent_biquadratic2}
    \rho^\ast s_4 = (d_1,d_2) \in \Br(K).
  \end{equation}
  Namely, the displayed formula right above (18) in \loccit~ states that
  $\rho^\ast s_4=\rho_1^\ast s_2+\rho_2^\ast s_2+(d_1, d_2)$. Here
  $\rho_i:\Gamma\to S_2$ is the homomorphism corresponding to $M_i/K$ and
  $s_2\in H^2(\Gamma, \{\pm 1\})$ is defined in \cite[\S 1.5]{SerreWitt}. In
  the middle of p.~654 in \loccit~it is shown that $s_2=0.$ This implies
  \eqref{eq:descent_biquadratic2}.

  Therefore Proposition \ref{prop:descent_rho} states that $\rho$ is induced by
  a divisor  $D\in \mathcal{S}$ if and only if the quadratic form
  \begin{equation}
       d_1 x^2 + d_2 y^2 - z^2 = 0
  \end{equation}
  has a nontrivial zero in $K$.
\end{example}

Once again let $D\in \mathcal{S}$. After replacing $D$ by an equivalent one, we
may write as $D = (g)_0$ for a special polynomial $g$. The group $H_D :=
\Aut_{\Kbar}(D)\subset\PGL_2(\Kbar)$ is conjugate to $\overline{G} =
\Aut_{\Kbar}(D_0)$ in $\PGL_2(\Kbar)$. The stabilizer in $\Aut_{\Kbar}(D)$ of
each root $\alpha$ of $g$ is a cyclic group of order $3$. Denote by
$\beta=\beta(\alpha)\in\PP^1_{\Kbar}$ the unique point different from $\alpha$
with the same stabilizer. We obtain a divisor $D'$ as the sum of the $\beta$
for $\alpha$ running through the roots of $g$. If no $\beta$ is equal to
$\infty$ then
\begin{equation}
  D'=(g')_0, \qquad \text{with } g'=\prod_\beta(y-\beta) \in K[y]
\end{equation}
Otherwise, $D'=(g')_0+\infty$, for a unique monic cubic polynomial $g'\in
K[y]$. In either case, we call the polynomial $g'$ the \emph{shadow} of $g$.

%We note that there
%exists an element $B\in\PGL_2(\Kbar)$ which normalizes $H_g$ and flips the set of
%roots of $g$ with that of its shadow $g'$. The element $B$ is unique up to
%composition with an element of $H'$, and the automorphism group of $D\cup D'$
%is isomorphic to $S_4$. In view of the construction of the homomorphism $\rho$
%induced by $g$, this implies the following statements.

\begin{lemma}\label{lem:descent_shadow}
  Let $g\in K[x]$ be a special polynomial with shadow $g'$, and let
  $\rho:\Gamma\to S_4$ (resp.\ $\rho'$) be the homomorphism corresponding to
  $D=(g)_0$ (resp.\ $D'=(g')_0$). Then $\rho'$ can be obtained by composing
  $\rho$ with an inner automorphism of $S_4$ given by conjugation with an odd
  element of $S_4$.

  In particular, given an extension $M$ of $K$, there are at most 2
  nonequivalent special polynomials over $K$ whose splitting field equals $M$.
\end{lemma}

\begin{proof}
  The divisor $D'=(g')_0$ is stabilized by $H_D$ by construction. Let $N_D$ be
  the normalizer of $H_D$ in $\PGL_2 (\Kbar)$. Then $N_D$ is isomorphic to
  $S_4$. Let $B$ be an element of $N_D \setminus H_D$. Then $B$ flips the two
  divisors $D$ and $D'$. In particular, the divisor $D \cup D'$ has
  automorphism group $N_D\cong S_4$ over $\Kbar$.

  Note that it suffices to check the claims above for the divisor $D_0$, where
  \begin{equation}
     D_0' = (y^4-8y)_0 = \{0, 2, 2\zeta_3, 2\zeta_3^2\}
  \quad \text{ and }
      B = \begin{pmatrix} 0 & -2 \\ 1 & 0 \end{pmatrix}.
  \end{equation}

  The statement of the lemma follows from the above, the construction of the
  homomorphism $\rho$ induced by $D$, and the fact that all automorphisms of
  $S_4$ are inner.
\end{proof}

\begin{proposition}\label{prop:pol_list}
  There are exactly $26$ nonequivalent special polynomials over $\QQ$ with good
  reduction away from $2,3$:
  \begin{align}\label{eq:pol_list1}
      x^4 + x, & \quad                                              \\
      x^4 + 2x,                  &\quad x^3 - 2,                       \\
      x^4 + 3x,                  &\quad x^3 - 3,                       \\
      x^4 + 6x,                  &\quad x^3 - 6,                       \\
      x^4 + 12x,                 &\quad x^3 - 12,                      \\
      x^4 - 6x^2 - 3,            &\quad x^4 + 6x^2 - 3,
                                              \label{eq:pol_list4} \\
      x^4 - 12x^2 - 12,          &\quad x^4 + 12x^2 - 12,
                                              \label{eq:pol_list5} \\
      x^4 + 6 x^2 + 8 x - 3,     &\quad x^4 + 4 x^3 - 6 x^2 - 4 x - 7, \\
      x^4 - 24x^2 + 32x - 48,    &\quad x^4 - 4x^3 + 24x^2 - 16x - 32, \\
      x^4 + 12 x^2 - 8 x - 12,   &\quad x^4 - 2 x^3 - 12 x^2 + 4 x - 14, \\
      x^4 - 36 x^2 + 96 x - 108, &\quad x^4 - 8 x^3 + 36 x^2 - 48 x - 12, \\
      x^4 + 12 x^2 + 64 x - 12,  &\quad x^4 + 16 x^3 - 12 x^2 - 32 x - 140, \\
      x^4 + 12 x^2 - 16 x - 12,  &\quad x^4 - 4 x^3 - 12 x^2 + 8 x - 20,
                                              \label{eq:pol_list2}  \\
      x^4 - 12x^2 + 32x - 12.
                                              \label{eq:pol_list3}
  \end{align}
  The polynomial on the right is equivalent to the shadow of the polynomial on the left.
\end{proposition}

\begin{proof}
  We apply Proposition \ref{prop:descent_rho} to the case where $K=\QQ$ and
  $L/\QQ$ is the maximal algebraic extension unramified outside $2,3$.  We
  start by listing the corresponding classes of homomorphisms $\rho:\Gamma\to
  S_4$ satisfying \eqref{eq:descent_rho_condition} up to conjugation by $S_4$.
  Such a homomorphism corresponds to a tuple $(M_1,\ldots,M_r)$ of finite
  extensions $M_i/\QQ$ that are unramified away from $2,3$ and such that
  $\sum_i [M_i:\QQ]=4$.

  Equation \eqref{eq:chi} implies that Condition
  \eqref{eq:descent_rho_condition} is equivalent to
  \begin{equation} \label{eq:descent_rho_condition2}
    \prod_i d(M_i) \equiv -3 \pmod{(\QQ^\times)^2} ,
  \end{equation}
  where $d (M)$ denotes the discriminant of a field $M$. In this setup the
  Galois group of the composite extension $M:=M_1\cdots M_r/\QQ$ is identified
  with a subgroup of $S_4$. Condition \eqref{eq:descent_rho_condition2} implies
  that the field of invariants of $M$ under $\Gal(M/\QQ)\cap A_4$ is
  $\QQ(\zeta_3)$.

  Of course, we may ignore the occurrences of $M_i=\QQ$. Ordering the subfields
  $M_i$ by their degrees, we either obtain one field $M/\QQ$ of degree $2,3,4$,
  or a pair $(M_1,M_2)$ of quadratic number fields.

  A consultation of the \emph{Database of Number Fields}, \cite{DatabaseNF},
  shows that there are precisely $11$ possibilities in the former case of a
  single number field $M$. More precisely, there is exactly one quadratic
  extension $M/\QQ$ unramified outside $2,3$ such that $d(M)\equiv -3$, namely
  \begin{equation}\label{eq:field1}
      M = K[\zeta_3] \quad \text{(with generating polynomial $x^2+x+1$).}
  \end{equation}
  There are precisely $4$ cubic number fields with this property, given by the
  generating polynomials
  \begin{equation}
     x^3-2, \quad x^3-3, \quad x^3 - 6, \quad x^3 - 12,
  \end{equation}
  and finally there are $8$ quartic fields, described by
  \begin{equation}\label{eq:field3}
    \begin{split}
      x^4 - 6 x^2 - 3, \quad x^4 - 12 x^2 - 12, \quad x^4 + 6 x^2 + 8 x - 3, \quad x^4 - 24 x^2 + 32 x - 48, \\
      x^4 + 12 x^2 - 8 x - 12, \quad x^4 - 36 x^2 + 96 x - 108, \quad x^4 + 12 x^2 + 64 x - 12, \quad x^4 + 12 x^2 - 16 x - 12 .
    \end{split}
  \end{equation}

  Lemma \ref{lem:descent_shadow} implies that each of these options corresponds
  to at most two special polynomials defining elements of $\mathcal{S}$. If
  there are two, then one is the shadow of the other and conversely. A simple
  search yields one pair $(g,g')$ for each of the possibilities in
  \eqref{eq:field1}--\eqref{eq:field3}.  These polynomials are listed as
  \eqref{eq:pol_list1}--\eqref{eq:pol_list2} above. As can be verified using
  the methods in \cite{LRS-ANTS}, the polynomial in \eqref{eq:pol_list1} is the
  only polynomial from this list that is $\PGL_2(K)$-equivalent to its own
  shadow.  It follows from Proposition \ref{prop:descent_rho} and Lemma
  \ref{lem:descent_shadow} that every element of the set $\mathcal{S}$ that is
  either irreducible or has at least one rational point is
  $\PGL_2(K)$-equivalent to a set of roots of a polynomial in this list.

  In the latter case of a pair of quadratic number fields $(M_1, M_2)$, the
  possibilities are
  \begin{equation}\label{eq:biquadratic_fields}
    (x^2+1, x^2-3), \quad (x^2-2, x^2+6), \quad (x^2+2, x^2-6).
  \end{equation}
  In the notation of Example \ref{exa:biquadratic} these pairs correspond to
  the Hilbert symbols
  \begin{equation}
      (d_1,d_2) \in \{ (-1, 3),\; (2, -6),\; (-2, 6)\}.
  \end{equation}
  A simple calculation with Hilbert symbols shows that the elements $(-1,3),(2,
  -6)\in\Br(\QQ)$ are nontrivial, whereas $(-2, 6) = 0$ in $\Br(\QQ)$. We
  conclude that only the pair $(-2, 6)$ can occur. Again a search yields the
  polynomial $g$ in line \eqref{eq:pol_list3} with splitting field
  $\QQ(\sqrt{-2},\sqrt{6})$. The shadow polynomial $g'$ is equivalent to $g$,
  and is therefore not listed.  Now our list is complete.
\end{proof}

\begin{theorem} \label{thm:special_2_3}
  Up to isomorphism, there are 800 nonisomorphic special Picard curves over
  $\QQ$ with good reduction outside $p=2,3$.
\end{theorem}

\begin{proof}
  Let $Y$ be a special Picard curve. The discussion above shows that $Y$ admits
  an equation
  \begin{equation}
    Y :\; x^4 = a g (y),
  \end{equation}
  where $g$ is one of the polynomials in Proposition \ref{prop:pol_list} and
  where $a = \pm 2^\mu 3^\nu$, with $0\leq \mu,\nu\leq 3$. By the same
  proposition, the polynomial $g$ is uniquely determined by $Y$.

  It remains to see, given such a special polynomial $g$, when $Y$ is
  equivalent to another curve of the form $Y' : x^4 = a' g (y)$ for $a = \pm
  2^{\mu'} 3^{\nu'}$. For this to happen with $\mu \neq \mu'$ or $\nu \neq
  \nu'$, the class of $g$ needs to admit nontrivial $K$-automorphisms, as
  described in Remark \ref{rem:special_isos}. Using \cite{LRS-ANTS}, there
  turns out to be a single nontrivial automorphism for the polynomials
  \eqref{eq:pol_list1}, \eqref{eq:pol_list4}, \eqref{eq:pol_list5}, and
  \eqref{eq:pol_list3}. In the cases \eqref{eq:pol_list4} and
  \eqref{eq:pol_list5}, this automorphism scales the polynomial $g$ itself by a
  $4$th power, so that we still cannot have $Y \cong Y'$ unless $a = a'$.
  However, in the cases \eqref{eq:pol_list1} and \eqref{eq:pol_list3}, the
  nontrivial automorphism multiplies $g$ by a factor in $9 (\QQ^\times)^4$.
  Therefore $Y$ and $Y'$ are isomorphic if and only if the quotient of $a$ and
  $a'$ is in either $(\QQ^\times)^4$ or $9 (\QQ^\times)^4$.

  This means that in $24$ of the $26$ cases of Proposition \ref{prop:pol_list}
  we get $32$ distinct curves for the values $a = \pm 2^\mu 3^\nu$, whereas in
  $2$ cases, we obtain only $16$ distinct curves. All in all we obtain $24
  \cdot 32 + 2 \cdot 16 = 800$ twists.
\end{proof}

\subsection{Discriminant minimization}

In this section we briefly mention discriminant minimization of equations of
special Picard curves over a field $K$, as a variant of the considerations in
Section \ref{sec:mindisc}.

\begin{definition}
  Let $Y$ be a special Picard curve over $K$. We call an equation for $Y$ of
  the form
  \begin{equation}\label{eq:specminshort}
    Y :\; b x^4 = c_0 y^4 + c_1 y^3 + c_2 y^2 + c_3 y + c_4
  \end{equation}
  a \emph{short Weierstrass equation} for $Y$.
\end{definition}

\begin{remark}
  The discriminant of the binary form corresponding $F = b x^4 - (c_0 y^4 + c_1
  y^3 + c_2 y^2 + c_3 y + c_4)$ corresponding to equation
  \eqref{eq:specminshort} is
  \begin{equation}
    \Delta (F) = -2^{16} b^9 \Delta (f),
  \end{equation}
  where $\Delta (f)$ is the discriminant of the univariate polynomial $f = c_0
  y^4 + c_1 y^3 + c_2 y^2 + c_3 y + c_4$.
\end{remark}

By Theorem \ref{thm:special_sweq} every special Picard curve over $K$ admits a
short Weierstrass equation. Conversely, by part (b) of the same theorem, a
nonsingular curve over $K$ with defining equation \eqref{eq:specminshort} is a
special Picard curve if and only if the binary quartic invariant
\begin{equation}\label{eq:invI}
  I = 12 c_0 c_4 - 3 c_1 c_3 + c_2^2
\end{equation}
vanishes \cite{cremona-fisher}.

\begin{remark}\label{rem:specialiso}
  A given plane quartic curve $Y$ over $K$ is a special Picard curve if and
  only if its Dixmier--Ohno invariants coincide with those of the standard
  special Picard curve \eqref{eq:special}. This yields an effective algorithm
  to decide whether $Y$ is special. To find an equation of the form
  \eqref{eq:specminshort} for $Y$, one calculates the automorphism group of $Y$
  and splits the ambient projective plane, or more canonically $\PP^2 H^0 (Y,
  \Omega_Y)$,  into the eigenspaces \eqref{eq:decomp} to obtain coordinates $x$
  and $y,z$.

  As mentioned in Remark \ref{rem:special_isos}, given two equations of special
  Picard curves $Y_i : x^4 = g_i (y)$ over $K$, finding isomorphisms between
  $Y_1$ and $Y_2$ reduces to determining equivalences of binary forms (up to
  scalars) between $g_1$ and $g_2$. For this question, too, effective
  algorithms exist \cite{LRS-ANTS}.
\end{remark}

To obtain long Weierstrass equations, we have to work slightly harder this time
around. First we define the appropriate notion.

\begin{lemma}\label{lem:specminlong}
  Let $Y$ be a plane curve over $K$ defined by a nonsingular equation of the
  form
  \begin{equation}\label{eq:speclongweq}
    Y :\; a_0 x^4 + a_1 (z) x^3 + a_2 (z) x^2 + a_3 (z) x = a_4 (y, z) ,
  \end{equation}
  where $a_i$ is a homogeneous form of degree $i$ in $y$ and $z$. Suppose that
  $Y$ is obtained from an equation of the form \eqref{eq:specminshort}. Then
  the equations
  \begin{align}
    8 a_0 a_2 & = 3 a_1^2 , \label{eq:specrel1} \\
    16 a_0^2 a_3 & = a_1^3 \label{eq:specrel2}
  \end{align}
  are satisfied.
\end{lemma}

\begin{proof}
  This follows from a direct calculation.
\end{proof}

\begin{definition}
  Given a special Picard curve $Y$ over $K$, we call an equation for $Y$ of the
  form \eqref{eq:speclongweq} a \emph{long Weierstrass equation} for $Y$.
\end{definition}

\begin{proposition}\label{prop:specminlong}
  Let $Y$ be a special Picard curve over a discretely valued field $(K, v)$.
  Then $Y$ admits an integral long Weierstrass equation of minimal discriminant
  exponent.
\end{proposition}

\begin{proof}
  Consider an integral short Weierstrass equation \eqref{eq:specminshort} for
  $Y$ in $\PP^2 (x : y : z)$. Let $F = b x^4 - c_0 y^4 - c_1 y^3 z - c_2 y^2
  z^2 - c_3 y z^3 - c_4 z^4$ the corresponding ternary form. We use the basis
  $x, y, z$ to identify $K^3$ with $K x \oplus K y \oplus K z$. Let $F_0 \in
  \OO_K [x,y,z]$ be an integral form of minimal discriminant defining a curve
  that is $K$-isomorphic to $Y$, and let $T_0 \in M_3 (\OO_K)$ be such that up
  to a scalar $F \cdot T_0 = F_0$ under the right action of $M_3 (\OO_K)$.

  \textbf{Claim 1:} There exists a matrix $T_1 \in \GL_3 (\OO_K)$ such that
  $T_0 T_1$ maps the subspace $K y \oplus K z$ to itself.

  \emph{Proof:} Consider the intersection $M = (V_2 \cdot T_0) \cap \OO_K^3$.
  Then $M$ is a torsion-free $\OO_K$-submodule of $\OO_K^3$ of rank $2$.
  Because $\OO_K$ is a principal ideal domain, we can find a complementary
  submodule of $M$ inside $\OO_K^3$. On the level of matrices, this comes down
  to saying that there exists $U_1 \in \GL_3 (\OO_K)$ whose second and third
  rows generate $M$. We can then take $T_1 = U_1^{-1}$ to prove the claim.

  The matrix $T_0 T_1$ now has the form
  \begin{equation}\label{eq:specspec}
    T_0 T_1 =
    \begin{pmatrix}
      * & b & a \\
      0 & * & * \\
      0 & * & * \\
    \end{pmatrix}
  \end{equation}

  \textbf{Claim 2:} There exists a matrix $T_2 \in \GL_3 (\OO_K)$ such that
  $T_0 T_1 T_2$ is of the form \eqref{eq:specspec} with $b = 0$.

  \emph{Proof:} If $b = 0$, we are done, and if $a = 0$, then we can take $T_2$
  to be the matrix corresponding to the transformation sending $(x, y, z)$ to
  $(x, z, y)$. For the same reason, we may otherwise suppose that $v (b) > v
  (a)$. But in that case we can take $T_2$ to correspond to $(x, y, z) \mapsto
  (x, y, (-b/a) y + z)$. The claim is proved.

  To conclude, let $U = T_0 T_1 T_2$. Since $T_1$ and $T_2$ are in $\GL_3
  (\OO_K)$, the matrix $U$ still has the property that a scalar multiple $G$ of
  $F \cdot U$ has minimal discriminant. Because of the form of $U$, the ternary
  quartic $G$ yields an equation \eqref{eq:speclongweq}.
\end{proof}

\begin{remark}
  As in Remark \ref{rem:global}, the same considerations apply globally over a
  number field with trivial class group: the step in Claim 2 can then be
  replaced by repeated division with remainder.
\end{remark}

\begin{theorem}\label{thm:specminshort}
  Let $Y$ be special Picard curve over a discretely valued field $(K, v)$ whose
  residue characteristic does not equal $2$. Then $Y$ admits an integral short
  Weierstrass equation \eqref{eq:specminshort} of minimal discriminant
  exponent.
\end{theorem}

\begin{proof}
  Extend $v$ to the Gauss valuation for ternary polynomials, and consider an
  integral equation \eqref{eq:speclongweq} of $Y$ whose discriminant is
  minimal. First suppose that $v (a_0) = 0$. Then the integral coordinate
  change $y \mapsto y - (a_1 / (4 a_0))$ gives an equation of the form
  \eqref{eq:minshort}, and we are done. If $v (a_0) = 1$, then $v (a_1) \geq 1$
  by \eqref{eq:specrel2}, and the same argument applies. If $v (a_0) = 2$, then
  similarly $v (a_1) \ge 2$ by \eqref{eq:specrel2}.

  Finally, suppose that $v (a_0) \geq 3$. Then $v (a_1) \ge 2$ by
  \eqref{eq:specrel2}, which also yields $v (a_0) \le (3/2) v (a_1)$. If $v
  (a_2) = 0$, then \eqref{eq:specrel1} yields  $v (a_0) \ge 2 v (a_1)$, a
  contradiction since $v (a_1) > 0$. This implies $v (a_2) \ge 1$. Knowing
  this, we can apply the same argument as at the end of the proof of Theorem
  \ref{thm:minshort}.
\end{proof}

\begin{example}\label{exa:Deltamin2}
  Let $Y_0$ be the standard special Picard curve defined by \eqref{eq:special}.
  This model has discriminant $2^{16} 3^9$ and is minimal except for $p=2$. A
  long Weierstrass equation \eqref{eq:speclongweq} which is minimal for all $p$
  is
  \begin{equation}
       2 x^4 - 4 x^3 + 3 x^2 - x = y^3,
  \end{equation}
  which one also recognizes as a short Weierstrass equation of the previously
  considered form \eqref{eq:minshort}. It has discriminant $2^7 3^9$.
\end{example}

\section{Comparing the conductor and the discriminant}
\label{sec:conductor}

In this section we strengthen some results from \cite{Picard1} on the exponent
of $p$ in the conductor of a Picard curve. Our main result Theorem
\ref{thm:mincond} states that the standard special Picard curve $Y_{0,\QQ}$ has
the smallest conductor among all special Picard curves defined over $\QQ$.  Our
method is based on the results of \cite{superell} on computing the conductor of
a superelliptic curve via stable reduction. The results of this section
illustrate that this method allows us to analyze the effect of twisting the
curve on the conductor. In Section \ref{sec:upperbound} we discuss the question
of comparing the conductor with the minimal discriminant of a superelliptic
curve.

\subsection{Calculating the conductor  via stable reduction}
\label{sec:conductor1}

We recall from \cite[Section 2]{superell} some facts on the conductor of a
curve over a number field and its relation to  stable reduction.  More details
and references to the literature can be found there.

Let $K$ be a number field and $Y/K$  a Picard curve. The \emph{conductor} of
$Y/K$ is an ideal
\begin{equation}
  \mathfrak{c}=\prod_{\p} \p^{f_\p},
\end{equation}
where the product runs over the prime ideals of $\mathcal{O}_K$.  The
\emph{conductor exponent} $f_{\p}$ is trivial if $Y$ has good reduction at
$\p$.

Denote by $K_{\p}^{\rm nr}$ the maximal unramified extension of the completion
of $K$ at $\p$.  The conductor exponent $f_{\p}$ measures the ramification of
the representation of $I_\p:=\Gal(\overline{K}_{\p}/K_{\p}^{\rm nr})$ acting on
the \'etale cohomology group $H^1_{\rm et}(Y_{\Kbar}, \QQ_\ell)$ for some
auxiliary prime $\ell$ different from the characteristic $p$ of $k=\OO_K/\p$.
In the rest of this section, we consider $Y$ as a curve over $K_{\p}^{\rm nr}$.
We drop $\p$ from the notation and write $K$ instead of $K_{\p}^{\rm nr}$.

Let $L/K$ be a Galois extension such that $Y_L$ has stable reduction over $L$
and let $\mathcal{Y}$ be the stable model of $Y_L$ over $\Spec(\OO_L)$. Define
$\Gamma:=\Gal(L/K)$.  For $u\geq 0$ we denote by $\Gamma^u$ the higher
ramification groups in the upper numbering. Since we assume that $k$ is
algebraically closed, we have that $\Gamma=\Gamma^0$, and the reduction
$\overline{Y}$ is a stable curve over $k$ on which the arithmetic Galois group
$\Gamma$ acts $k$-linearly.

The quotient curve $\overline{Y}^u := \overline{Y}/\Gamma^u$ is again
semistable.  The following proposition is \cite[Prop.~2.3]{Picard1} and follows
from \cite[Theorem 2.9]{superell} and \cite[Corollary 2.14]{ICERM}.

\begin{proposition}\label{prop:fp}
  The conductor exponent of the curve $Y/K$ is given by
  \begin{equation} \label{eq:fp1}
    f_\p = \epsilon + \delta,
  \end{equation}
  where
  \begin{equation} \label{eq:fp2}
    \epsilon := 6 - \dim \, H^1_{\rm et}(\Yb^0,\QQ_\ell)
  \end{equation}
  and
  \begin{equation} \label{eq:fp3}
    \delta := \int_{0}^\infty \left(6-2g(\Yb^u)\right) {\rm d}u.
  \end{equation}
  In particular, $\delta=0$ if and only if $Y$ acquires stable reduction over a
  tamely ramified extension.
\end{proposition}

We say that the stable reduction is \emph{of compact type} if its dual graph is
a tree. This is equivalent to the Jacobian of $Y$ having potentially good
reduction. In the case that $\Yb^u$ is a semistable curve of compact type, we
have that $\dim\, H^1_{\rm et}(\Yb^u,\QQ_\ell)=2g(\Yb^u)$, and $g(\Yb^u)$ is
the sum of the genera of the irreducible components of the normalization of
$\Yb^u$. In the general case, one needs to add the number of loops of the dual
graph of $\Yb^u$. We refer to \cite[Section 2.2]{Picard1} for a precise
formula.

\begin{remark}\label{rem:fp_bound}
  Let $Y/\QQ_3^\nr$ be a Picard curve given as superelliptic curve of exponent
  $3$ \eqref{eq:shortweq}.  In \cite[Theorem 3.2]{Picard1} the possibilities
  for the stable reduction are determined and \cite[Theorem 3.6]{Picard1}
  yields a lower bound on $f_3$ for each of the possibilities. The statement is
  \begin{enumerate}[(i)]
    \item $f_3\geq 6$ if $Y$ has potentially good reduction (Case (a)),
    \item $f_3\geq 4$ if $Y$ does not have potentially good reduction, and
      $\Yb$ is of compact type (Cases (b) and (c)),
    \item $f_3\geq 5$ if the stable reduction of $\Yb$ has loops
      (Cases (d) and (e)).
  \end{enumerate}
  Here the cases are as in \cite[Theorem 3.2]{Picard1}.  All lower bounds are
  attained. One may check that these lower bounds also apply for special Picard
  curves defined by a superelliptic equation of exponent $3$
  \eqref{eq:shortweq} over $\QQ_3^\nr$.

% An example with $f_3=4$ is the curve  defined by $y^3 = x^4 + x^3
%  + 3^5x$ (\cite[Example 3.9]{Picard1}). It achieves stable reduction
%  at $p=3$ over a tame extension, hence $\delta=0$. Its stable
%  reduction at $p=3$ consists of two irreducible components, one of
%  genus $2$ and one of genus $1$, intersecting in one point. This is
%  Case (b) of \cite[Theorem 3.2]{Picard1}.  Considered as smooth
%  projective curve over $\QQ$ we find $N=2^43^55^4$.

 % An example with $f_3=5$ is the curve defined by $y^3 = x^4 + x^3 +
 % 3^5x$. It has $\delta=0$.  Its stable reduction at $p=3$ consist of
 % two irreducible components of genus $1$ and $0$, respectively,
 % intersecting in $3$ points. This is Case (d) of \cite[Theorem
 %   3.2]{Picard1}.

  In \cite[Theorem 3.6.(a)]{Picard1} it is mistakenly claimed that if $f_3\leq
  6$ then $Y$ achieves stable reduction over a tamely ramified extension. A
  counterexample is given by the curve
  \begin{equation}\label{eq:f3=5}
    Y:\; y^3 = x^4 + x^3 + 3^3x^2 +3^5x.
  \end{equation}
  For this curve we find $\epsilon=4$ and $\delta=1$, and hence $f_3=4+1=5$.
  The stable reduction of $Y$ is in Case (b) of \cite[Theorem 3.2]{Picard1}.

  This example illustrates that the statement of \cite[Theorem
  3.6.(c)]{Picard1} needs to be modified, as well. The correct statement is
  that if $\delta=5$ then the stable reduction at $p=3$ contains loops (Cases
  (d) or (e)). If $f_5=5$ then $(\epsilon, \delta)\in \{(5,0), (4,1)\}$, and
  both possibilities  occur.
\end{remark}

In the rest of this section we consider special Picard curves.  We start by
treating the case of residue characteristic $p=3$.

\begin{proposition}\label{prop:f3special}
  Let $Y/K:=\QQ_3^\nr$ be a special Picard curve given by an equation
  $x^4=ag(y)$ as in \eqref{eq:special2}. We have that $\epsilon\in \{4,6\}$ and
  $f_3\geq 4$.
\end{proposition}

\begin{proof}
  Let $L/K$ be a Galois extension such that $Y$ admits a stable model over
  $\OO_L$. Assume that $\zeta_3\in L$. Let $\mathcal{Y}\to \Spec(\OO_L)$ be the
  stable model of $Y_L$ and $\overline{Y}$ its special fiber. As in Section
  \ref{sec:red_special}, we write $\mathcal{W}=\mathcal{Y}/\langle\tau\rangle$
  and $\overline{W}$ for the special fiber of $\mathcal{W}$. Since $Y$ has
  potentially good reduction by \cite[Section 5.1.3]{MichelDiss},
  $\overline{Y}$ and $\overline{W}$ are smooth curves of genus $3$ and $0$,
  respectively.

  We write $\Gamma=\Gal(L/K)$ and $\Yb^0=\Yb/\Gamma$.  Proposition
  \ref{prop:fp} implies that $f_3\geq \epsilon\geq 6-2g(\Yb^0)$. We claim that
  $g(\Yb^0)\leq 1$. This implies that $\epsilon\in \{4,6\}$.

  Since $\zeta_3\notin K$ Lemma \ref{lem:specialp=3} implies that there exists
  an element $\gamma\in \Gamma$ that acts on $\overline{D}$ as an odd
  permutation.  One calculates the genus of the quotient
  $\Yb/\langle\gamma\rangle$ for all possibilities for the action of $\gamma$
  on $\overline{Y}$. Since $\gamma$ acts as an element of $\Aut_k(\Yb)$ this
  may be done on the reduction $\overline{Y}_0$ of the standard special Picard
  curve $Y_0$ \eqref{eq:special}. It follows that the only possibility for
  $g(\Yb/\langle \gamma\rangle)$ to be positive is that $\gamma$ acts as a
  $4$-cycle on $\overline{D}$ and fixed-point free on $\Yb$, but $\gamma^2$
  fixes exactly $4$ points. In this case $g(\Yb/\langle \gamma\rangle)=1$ and
  hence $g(\Yb^0)\leq 1$. In all other cases, we have that $g(\Yb^0)=0$. The
  statement of the proposition follows.
\end{proof}

The following example shows that the lower bound in Proposition
\ref{prop:f3special} is sharp.

\begin{example}\label{exa:specialp=3}
  We consider the special Picard curve over $\QQ$ defined by
  \begin{equation}
    Y_a:\;x^4=a g(y) \text{ with }g(y)=y^4+6by^2+3cy-3b^2 \text{ such that
    $3\nmid b$ and  $a\neq 0$}.
  \end{equation}
  It is no restriction to assume that $0\leq i:=v_3(a)\leq 3$.  Let
  $K=\QQ_3^{\nr}$ and $L=K[\pi]$ with $\pi^4=3$. Substituting $x=\pi^{i+1} x_1$
  and $y=\pi y_1$ and dividing both sides of the equation by $\pi^4$ yields a
  smooth $\OO_L$-model of $Y_L$. Its special fiber is a smooth projective curve
  of genus $3$ with affine equation
  \begin{equation}
    \overline{Y}:\; x_1^4=\overline{a}(y_1^4-\overline{b}^2),
  \end{equation}
  where $\overline{b}$ is the reduction of $b$ and $\overline{a}$ the reduction
  of $a/3^i$ modulo $\pi$.

  The arithmetic Galois group $\Gal(L/K)$ is generated by
  $\gamma(\pi)=\zeta_4\pi$. The $k$-linear automorphism on $\overline{Y}$
  induced by $\gamma$ satisfies: $(x_1, y_1)\mapsto (\zeta_4^{3-i} x_1,
  \zeta_4^3 y_1)$ and acts as a $4$-cycle on the reduction $\overline{D}$ of
  the branch locus of $\psi$.

  One computes that $g(\Yb/\Gamma)=1$ if and only if $v_p(a)\in \{1, 2\}$ and
  $g(\Yb/\Gamma)=0$ otherwise.  Since $L/K$ is tamely ramified, $f_3=\epsilon$.
  We conclude that $Y_a$ has conductor exponent $f_3=4$ if and only if
  $v_p(a)\in \{1, 2\}$. In the case that $v_p(a)\in \{0, 3\}$ one finds
  similarly that $f_3=6$.
\end{example}

The next proposition treats the case of residue characteristic $\neq 3$. The
corresponding result for Picard curves given by a superelliptic equation of
exponent $3$ \eqref{eq:shortweq} is \cite[Theorem 4.4]{Picard1}. In general it
is not true that $f_p=0$ if and only if $Y$ has good reduction to
characteristic $p$: in the case that the curve $Y$ has bad reduction but its
Jacobian variety has good reduction to characteristic $p$ we also have that
$f_p=0$.  An example is given in \cite[Example 5.5]{Picard1}. However, this
does not occur for special Picard curves.

\begin{proposition}\label{prop:f2special}
  Let $Y/K=\QQ_p^{\nr}$ be a special Picard curve.
  \begin{enumerate}[(a)]
    \item If $p=2$ then $f_2\geq 6$.
    \item If $p\geq 5$ then $f_p\in \{0, 4, 6\}$. Moreover, $f_p=0$
      if and only if $Y$ has good reduction at $p$.
  \end{enumerate}
\end{proposition}

\begin{proof}
  Assume that we are given a defining equation $x^4=ag(y)$ for $Y$ of the form
  \eqref{eq:special2}. Let $g'$ be the shadow polynomial of $g$ as defined in
  Section \ref{sec:classification}.  Let $\alpha$ be a root of $g$. Lemma
  \ref{lem:special_aut}.(c) implies that there is a unique subgroup
  $H_\alpha\subset \Aut_{\Kbar}(Y)$ of order $3$ with $P_\alpha:=(0, \alpha)$
  as fixed point.  We denote by $y=\beta$ the second fixed point of $\sigma$
  considered as automorphism of $Y/\langle\tau\rangle =\PP^1_y$. In other
  words, $\beta$ is the root of the shadow polynomial $g'$ of $g$ corresponding
  to $\alpha$, see Section \ref{sec:classification}. The ramification locus
  $\mathcal{R}_\phi$ of $\phi:Y\to Y/H_\alpha$ is precisely the inverse image
  of $y=\beta$ on $Y$, together with $P_\alpha$. Here $P_\alpha$ is the
  distinguished point described in Remark \ref{rem:distpoint}.

  We apply the strategy of \cite[Sections 4+5]{superell} for computing the
  stable reduction of $Y$.  The stable reduction of the standard special Picard
  curve $Y_{0, K}$ is computed in \cite[Section 5.1.3]{MichelDiss}.  That
  calculation implies that to find all irreducible components of the stable
  reduction $\overline{Y}$ of the twist $Y$ of $Y_{0,K}$ it suffices to
  separate the points of $\mathcal{R}_\phi$.  We refer to \cite[Sections 4.2
  and 4.3]{superell} for an explanation of the procedure of separating the
  points. The structure of the stable reduction $\overline{Y}$ is as follows.
  \begin{equation*}
    \setlength{\unitlength}{0.7mm}
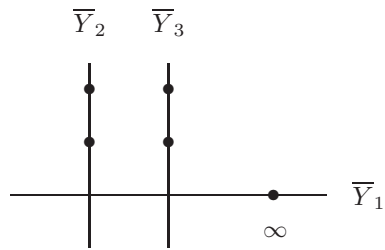
    \begin{picture}(80,50)
      \put(0,10){\line(1,0){60}}
      \put(15,0){\line(0,1){35}}
      \put(30,0){\line(0,1){35}}

      \put(15,20){\circle*{2.2}}
      \put(15,30){\circle*{2.2}}
      \put(30,20){\circle*{2.2}}
      \put(30,30){\circle*{2.2}}
      \put(50,10){\circle*{2.2}}

      \put(65,8){\smaller[1]$\Yb_1$}
      \put(12,41){\smaller[1]$\Yb_2$}
      \put(27,41){\smaller[1]$\Yb_3$}

      \put(6,28.5){\smaller[2]}
      \put(3,19){\smaller[2]}
      \put(35,28.5){\smaller[2]}
      \put(34,19){\smaller[2]}
      \put(48,2){\smaller[2]$\infty$}
    \end{picture}
  \end{equation*}
  Here the dots indicate the specialization of $\mathcal{R}_\phi$. The dot
  marked $\infty$ is the specialization of the distinguished point $P_\alpha$.
  All three irreducible components are smooth curves of genus $1$.

  \bigskip
  We first consider the case that $g$ has a $K$-rational root $\alpha$. In this
  case the subgroup $H_\alpha$, considered as subgroup of $\GL_2(\Kbar)$ as in
  \eqref{eq:Gmatrix}, is $K$-rational.  Arguing as in the proof of Theorem
  \ref{thm:shortweq} we conclude that $Y/K$ admits a defining equation
  \eqref{eq:shortweq} as superelliptic curve of exponent $3$. Since $Y$ is
  special it even follows that $Y$ admits an equation
  \begin{equation}
    y^3=ax^4+d.
  \end{equation}

  Arguing as in \cite[Section 5.1.3]{MichelDiss} we find that any Galois
  extension $L/K$ such that $Y_L$ has semistable reduction contains $\zeta_4$
  and $\sqrt[3]{2}$. Moreover, any element $\gamma\in \Gal(L/K)$ with
  $\gamma(\zeta_4)=-\zeta_4$ acts nontrivially on at least one of the two
  irreducible components $\overline{Y}_2, \overline{Y}_3$. Similarly, any
  element $\gamma\in \Gal(L/K)$ with $\gamma(\sqrt[3]{2})\neq \sqrt[3]{2}$ acts
  nontrivially on either $\overline{Y}_1$, or on $\overline{Y}_2$ and
  $\overline{Y}_3$, or on all three irreducible components. It follows
  therefore from Proposition \ref{prop:fp} that
  \begin{equation}
    \epsilon\geq 4, \qquad  \delta\geq 2
  \end{equation}
  and hence that $f_2=\delta+\epsilon\geq 6$. This proves  statement (a) in
  this case.

  \bigskip
  Assume that none of the roots of $g$ is $\QQ_2^{\nr}$-rational. It follows
  from the definition of $g'$ that none of the roots of $g'$ is
  $\QQ_2^{\nr}$-rational, as well. One computes using the explicit stable model
  of $Y_{0}$ from \cite[Section 5.1.3]{MichelDiss} that all $16$ points whose
  $y$-coordinate is a root of $g'$ specialize to pairwise distinct points of
  the stable reduction of $\overline{Y}$; exactly half specialize to
  $\overline{Y}_2$ and $\overline{Y}_3$, respectively. The automorphism group
  of these $16$ points is exactly the group $\tilde{S}_4$ discussed in Section
  \ref{sec:classification}. The subgroup $\tilde{A}_4$ of index $2$ stabilizes
  each of the two components $\overline{Y}_2$ and $\overline{Y}_3$.

  Let $L/K$ be a Galois extension over which $Y$ admits stable
  reduction. It is no restriction to assume that all $16$ points whose
  $y$-coordinate is a root of $g'$ are defined over $L$.  Let
  $\gamma\in \Gal(L/K)$ be an element that fixes none of the zeros
  $(g')_0$.  It follows from the description of $\tilde{A}_4$ in
  Remark \ref{rem:A4tilde} that the quotient of
  $\overline{Y}_2\cup\overline{Y}_3$ by $\gamma$ has genus $0$.  Here
  we use that $\gamma$ does not lift to an automorphism of order $2$
  and hence does not just permute the two components $\Yb_2$ and
  $\Yb_3$. Proposition \ref{prop:fp} implies in this situation that
  \begin{equation}
    \epsilon\geq 4, \qquad \delta\geq 4.
  \end{equation}
  As in the first case we conclude that $f_2=\delta+\epsilon\geq 4+4\geq 6$ and
  Statement (a) is proved.

  \bigskip We prove (b).  The curve $Y$ has potentially good reduction to
  characteristic $p\geq 5$. It follows that $\Yb^0 = \Yb/\Gamma$ is a smooth
  curve, and hence that $\epsilon = 2g(Y) - 2g(\Yb^0)\leq 6$ is even.  Lemma
  \ref{lem:special_aut} implies that $\Aut_k(\Yb) = \Aut_{\Kbar}(Y) = G_{48}$.
  A calculation using the description of this group in Lemma
  \ref{lem:special_aut}.(a) implies that $g(\Yb/H)\leq 1$ for all nontrivial
  subgroups $H<\Aut_k(\Yb)$. We conclude that $\epsilon\in \{0,4,6\}$.

  It follows from \cite[Cor.~4.6]{superell} that the stable model of $Y$ may be
  defined over a tame extension of $K=\QQ_p^\nr.$ This implies that $\delta=0$
  and hence that $f_p=\epsilon\in \{0,4,6\}$. This finishes the proof.
\end{proof}

The standard special Picard curve $Y_{0, \QQ}$ defined by \eqref{eq:special}
has conductor $N=2^63^6$. This is the smallest value in our database
\cite{bsw-code}. In \cite[Problem 5.7]{Picard1} it was  asked whether there
exists a Picard curve with strictly smaller conductor. As a consequence of the
results of Sections \ref{sec:special} and \ref{sec:conductor1}, we can at least
answer this question for special Picard curves.

\begin{theorem}\label{thm:mincond}
  Let $Y/\QQ$ be a special Picard curve. Then  $N_Y\geq 2^63^6$.
\end{theorem}

\begin{proof}
  Let $Y/\QQ$ be a special Picard curve and assume that $N_Y<2^63^6$.
  Propositions \ref{prop:f3special} and \ref{prop:f2special}.(a) imply that
  $N_Y\geq 2^63^4$. Therefore Proposition \ref{prop:f2special}.(b) implies that
  $Y$ has good reduction at all primes $\geq 5$. Moreover, it follows from
  Proposition \ref{prop:f3special} that $f_3<6$ and hence that the invariant
  $\epsilon_3$ equals $4$.

  Let $x^4=a g(y)$ be an equation for $Y$ of the form \eqref{eq:special2} and
  let $L_0/\QQ$ be the splitting field of $g$. Since $\epsilon_3=4<6$ the proof
  of Proposition \ref{prop:f3special} implies that the inertia group $I_3$ at
  $3$ of the extension $L_0/\QQ$ acts as a $4$-cycle on the divisor of zeros
  $D=(g)_0$.  The only  special polynomials found in  Proposition
  \ref{prop:pol_list} that satisfy this condition are
  \begin{equation}\label{eq:gs}
    g\in \{ y^4\pm 12y^2-12, y^4\pm 6y^2-3\}.
  \end{equation}
  It follows from Example \ref{exa:specialp=3} that $v_3(a)\in \{1,2\}$. In all
  $8$ cases we have indeed that $f_3=\epsilon_3=4<6$.

  We consider the conductor exponent at $2$.  The considerations at $3$ imply
  that $g$ is one of the polynomials in \eqref{eq:gs}. One checks that $g$ is
  irreducible over $\QQ_2^{\nr}$.  The proof of Proposition
  \ref{prop:f2special}.(a) therefore implies that the splitting field of $g'$,
  and hence of $g$, over $\QQ_2^{\nr}$ is contained in any Galois extension
  $L/\QQ_2^\nr$ over which $Y$ acquires good reduction. The Galois group of the
  splitting field of $g$ is $\tilde{A}_4$ for $g\in \{ y^4\pm 6y^2-3\}$ and
  $\tilde{S}_4$ for $g\in \{ y^4\pm 12y^2-12\}$. Computing the jumps in the
  filtration of the higher ramification groups of the splitting fields and
  applying Proposition \ref{prop:fp}, one finds that $f_2\geq 12$ in all cases.
  This contradicts the assumption that $N<2^63^6$ and the proof is finished.
\end{proof}

\begin{example}
  The curve
  \begin{equation}
    Y :\; x^4=12(y^4+6y^2-3)
  \end{equation}
  has conductor $N=2^{12}3^4$.
\end{example}

\subsection{An upper bound for the conductor exponent?}\label{sec:upperbound}

Let $(K, v)$ be a complete discrete valuation field of mixed characteristic
whose residue field $k$ is a perfect field of characteristic $p$. Sutherland
asked us whether for smooth plane curves $Y/K$ it is true that
\begin{equation}\label{eq:ineq}
  f_p\leq v_p(\Delta^{{\rm min}}).
\end{equation}
The corresponding formula holds for elliptic curves (where it follows from
Ogg's formula \cite{Ogg67}), and (with a suitable notion of \emph{minimal
discriminant}) for curves of genus $2$, by work of Liu \cite{liu94}.

The inequality \eqref{eq:ineq} would imply that plane curves with small
discriminant also have small conductor. This is useful when collecting curves
of fixed genus with small conductor, as was done in \cite{Sutherland_database}
for the case of plane quartics. Our results on discriminant minimization can be
used to prove that \eqref{eq:ineq} holds for all Picard curves and for residue
characteristic $\neq 2,3$. Details will appear in \cite{roman-thesis}.
Moreover, as detailed in Appendix \ref{sec:database}, our computations of
explicit examples give strong evidence that the result remains true for $p=2,3$
as well.

To elaborate somewhat, recall that every nonspecial Picard curve over $K$ has a
short Weierstrass equation \eqref{eq:minshort} by Theorem \ref{thm:shortweq}.
Moreover, if $p\neq 3$, then we may assume that the corresponding homogenous
equation has minimal discriminant exponent (Theorem \ref{thm:minshort}).
Similarly, a special Picard curve can be given by a short Weierstrass equation
\eqref{eq:specminshort} by Theorem \ref{thm:special_sweq}, and for $p\neq 2$ we
may once more assume that the corresponding homogenous equation has minimal
discriminant exponent (Theorem \ref{thm:specminshort}). In each case, the curve
$Y$ can be given as a superelliptic curve of exponent prime to $p$. In his
upcoming PhD thesis \cite{roman-thesis}, Roman Kohls proves upper bounds for
the conductor exponent at $p$ of general superelliptic curves of exponent prime
to $p$, generalizing results of Srinivasan \cite{padma} for hyperelliptic
curves. Applied to integral short Weierstrass equations of the form
\eqref{eq:minshort} and \eqref{eq:specminshort}, these bounds prove the
inequality \eqref{eq:ineq}, for all Picard curves and for $p\neq 2,3$.

\appendix
\numberwithin{equation}{section}

\section{A description of the database}\label{sec:database}

We have made a database of Picard curves available at \cite{bsw-code}. At the
moment of publication, this database contained 5678 isomorphism classes of
Picard curves over $\QQ$ that have bad reduction at only two primes in $\left\{
2, 3, 5, 7 \right\}$. Of these curves, 800 are the special Picard curves
described in Theorem \ref{thm:special_2_3}. The other curves were constructed
from input furnished by \cite{MalmskogRasmussen} and \cite{MichelDiss}, from
computations in the upcoming work \cite{kmr} that we describe in Appendix
\ref{sec:modified_Smart}, and from an exhaustive search conducted by Andrew
Sutherland \cite{Sutherland_database}.

We represent a Picard curve $Y$ in the database by a reduced polynomial
\eqref{eq:binredweq}, but the database also gives minimal long and short long
Weierstrass equations over $\ZZ$. It further includes the invariants of $Y$
over $\QQ$ and $\QQbar$ (Proposition \ref{prop:invs}), the factorization of its
discriminant (Definition \ref{def:mindisc}) and finally, in many cases, its
conductor along with its reduction type at bad primes (Section
\ref{sec:conductor}).

Because determining the geometric and arithmetic invariants of a given Picard
curve over $\QQ$ via Proposition \ref{prop:invs} is an efficient operation, it
is possible to quickly look up whether a nonspecial curve over $\QQ$ given by
the user is in the database. Similarly, for special curves, isomorphisms are
found by using fast algorithms exploiting the special configuration of
hyperflexes described in Remark \ref{rem:specialiso}. Finally, using
$\QQbar$-invariants allows us to quickly look up all the twists of a given
curve that are in the database.

For details of the implementation which is written in \textsc{Magma}
\cite{Magma} and \textsc{SageMath} \cite{SageMath}, we refer to the file
\texttt{README.md} at \cite{bsw-code}. For 3516 of the curves in the database,
we were able to compute the conductor using the MCLF package in
\textsc{SageMath} \cite{mclf}. In all these cases, the conductor exponent is
bounded by the exponent of the minimal discriminant at all primes, including
$2$ and $3$. For all but 224 of the other curves, we still managed to calculate
the conductor exponents away from the prime $3$ (for nonspecial curves) or $2$
(for special curves).

\section{Computations}\label{sec:modified_Smart}

Let $S$ be a finite set of primes containing $3$. By Proposition
\ref{prop:goodredbinary}, we can compute the set of all nonspecial Picard
curves with good reduction outside $S$ by computing the set of equivalence
classes of binary quartic forms whose discriminant is a $S$-unit. A method to
find these binary quartic forms was developed by Smart \cite{Smart97}; it
reduces this question to solving of $S$-unit equations. In this section, we
briefly explain how one can modify and improve Smart's method in order to
compute Picard curves with good reduction outside $S$ for $S = \{ 2, 3 \}$, $\{
3, 5 \}$ and $\{ 3, 7 \}$. We will give a detailed exposition of the method in
the forthcoming paper \cite{kmr}.

Let $Y/\QQ$ be a nonspecial Picard curve with good reduction outside $S$. By
Corollary \ref{cor:binred}, the curve $Y$ admits a reduced short Weierstrass
equation
\begin{equation}
  y^3 = c f(x),
\end{equation}
where $f(x) = x^4 + a_3 x^3 + a_2 x^2 + a_1 x + a_0$ and $c, a_i \in \ZZ$. Let
$K / \QQ$ be the splitting field of $f$, let $\alpha_1, \dots, \alpha_4$ be the
roots of $f$ in $K$, and let $G = \Gal(K / \QQ)$. Let $S_K$ be the set of prime
ideals over $S$, and let $\OKSu \subset K^*$ the $S$-unit group associated to
$S_K$.

By Proposition \ref{prop:goodredbinary}, we have $v_p (c) = v_p (\Delta (f)) =
0$ for all $p$ outside $S$. Since $\Delta (f) = \prod_{1 \leq i < j \leq 4}
(\alpha_i - \alpha_j)^2$ , we obtain
\begin{equation}
  \alpha_i - \alpha_j \in \OKSu.
\end{equation}
Given $i,j,k \in \{1,2,3,4\}$ with $i \neq j,k$ and $j \neq k$, we define the
\emph{cross ratio}
\begin{equation}
  \lambda_{i,j,k} = \frac{\alpha_j - \alpha_i}{\alpha_k - \alpha_i} .
\end{equation}
Then for all $i, j, k$ as above we have that $\lambda_{i,j,k} \in \OKSu$ and $1
- \lambda_{i,j,k} \in \OKSu$. This means that $\lambda_{i,j,k}$ is a solution
of the $S$-unit equation
\begin{equation}\label{eq:Sunit}
  \lambda + \mu = 1,
\end{equation}
where $\lambda, \mu\in\OKSu$.

Given the values $\lambda_{i,j,k}$, the roots $\alpha_i$ can be calculated
using an argument involving Hilbert's Theorem 90. This gives us the
$\QQbar$-isomorphism class of $Y$, from which we can determine all models
\eqref{eq:binredweq} with good reduction outside $S$ by using the invariants
from Section \ref{sec:invariants} along with the description of twists of
nonspecial Picard curves in \cite{Lorenzo}.

Methods for solving $S$-equations were developed by De Weger, Tzanakis, Smart
and others \cite{TzanakisWeger89,Smart97}. The complexity of their algorithms
increases exponentially with the rank of $\OKSu$; we typically cut it down by
restricting to a subgroup of $\OKSu$, as in \cite{Angelos}. Moreover, this
complexity depends on the representation of the Galois group $G$ on the roots
of $f$. Currently, our database contains partial results for $S$ equal to $\{
2, 3 \}$, $\{ 3, 5 \}$ and $\{ 3, 7 \}$ and $G$ acting either trivially or as
$\langle (1,2) \rangle$ or $\langle (1,2,3)\rangle$.

\end{document}